\documentclass[12pt]{article}
\usepackage[margin=1in]{geometry}

\usepackage{amsmath, amsfonts,amssymb, mathtools}
\usepackage{amsthm}
\usepackage{enumerate, bbm}
\usepackage{graphicx}
\usepackage{authblk}
\usepackage{multirow}
\usepackage{booktabs}
\usepackage{xcolor}

\usepackage{natbib}
\usepackage[colorlinks,citecolor=blue,urlcolor=blue]{hyperref}

\usepackage{caption}
\usepackage{subcaption}
\usepackage{enumitem}

\usepackage{setspace}
\usepackage{parskip}
\newtheoremstyle{paperthm}
  {1em}   
  {1em}   
  {\itshape} 
  {}      
  {\bfseries} 
  {.}     
  {.5em}  
  {}      

\theoremstyle{paperthm}

\newtheorem{theorem}{Theorem}
\newtheorem{lemma}{Lemma}
\newtheorem{assumption}{Assumption}
\newtheorem{corollary}{Corollary}
\newtheorem{proposition}{Proposition}
\newtheorem{definition}{Definition}

\newtheorem{remark}{Remark}

\newcommand{\E}{\mathbb E}
\newcommand{\Pp}{\mathbb P}
\newcommand{\R}{\mathbb R}
\newcommand{\cR}{\mathcal R}
\newcommand{\cA}{\mathtt A}
\newcommand{\cB}{\mathtt B}
\newcommand{\cBD}{\mathtt{BD}}
\newcommand{\cM}{\mathcal M}
\newcommand{\cI}{\mathcal I}
\newcommand{\cJ}{J}

\newcommand{\ip}[2]{\left\langle #1,#2\right\rangle}
\newcommand{\norm}[1]{\left\lVert #1\right\rVert}
\newcommand{\abs}[1]{\left\lvert #1\right\rvert}

\newcommand{\Ht}{\mathcal H^s(L,p_{\max})}

\newcommand{\dd}{d}
\newcommand{\Var}{\operatorname{Var}}
\newcommand{\intm}{\mathrm{int}}
\newcommand{\sml}{\mathrm{small}}

\DeclarePairedDelimiterX{\infdivx}[2]{(}{)}{%
  #1\;\delimsize\|\;#2%
}

\DeclarePairedDelimiter\floor{\lfloor}{\rfloor}

\begin{document}

\title{Minimax Rates for Learning \\ Smooth Populations of Parameters}
\author[1]{JungHo Lee\thanks{\url{junghol@stat.cmu.edu}}}
\author[1]{Edward H. Kennedy\thanks{\url{edward@stat.cmu.edu}}}
\affil[1]{Department of Statistics and Data Science, Carnegie Mellon University}

\date{\today}


\maketitle

\begin{abstract}
    We study pointwise estimation of a smooth binomial mixing density and characterize its minimax rates. Assuming that the mixing density is $s$-H\"older smooth, we first consider the homogeneous setting with a common number of binomial trials $t$. We derive matching lower and upper bounds that reveal three regimes depending on the relative sizes of the sample size $n$ and the number of trials $t$. When $t$ is small, the finite number of trials creates an identification barrier that persists regardless of sample size; at intermediate $t$, more data reduce sampling uncertainty, while the number of trials limits how finely the mixing density can be recovered; and when $t$ is sufficiently large, the usual nonparametric density estimation rate is recovered. Our lower bounds exploit the polynomial structure of the binomial mixture model, while attainability is achieved by a suitably regularized orthogonal series estimator. We then extend the analysis to heterogeneous trial parameters, where the minimax rate depends on the full trial profile through the effective sample size available at each polynomial degree. These results characterize the fundamental statistical limits of learning smooth populations of binomial probabilities under both homogeneous and heterogeneous trials.
\end{abstract}

\vfill

\newpage
\section{Introduction}
In this paper, we consider the problem of estimating the density of a population of latent binomial probabilities. We observe independent and identically distributed samples generated according to
\[
    X_i\mid Q_i=q\sim\operatorname{Bin}(t,q),
    \qquad
    Q_i\overset{\mathrm{iid}}{\sim}p,
    \qquad i=1,\ldots,n,
\]
where the mixing density $p$ is unknown and supported on $[0,1]$. While seemingly simple, the problem is nontrivial because binomial sampling fundamentally limits identifiability; with fixed trial counts, the observed data reveal only limited features of $p$, and distinct mixing distributions can therefore generate the same observed law. This issue has long motivated work on identifiability and estimation for binomial mixtures including \cite{teicher1963identifiability, blischke1964estimating, lord1969estimating, wood1999binomial}.

The lack of unrestricted identification does not rule out useful estimation, but it makes the choice of structure or regularization central. At one extreme, one may posit a low-dimensional parametric family, e.g., a beta family, and reduce the problem to maximum likelihood estimation. A more flexible empirical Bayes approach by \citet{walter1987empiric} uses a truncated Jacobi polynomial expansion around a beta reference density, regularizing the problem by retaining only finitely many polynomial components. A different and particularly simple strategy, especially when the trial counts are large, is to treat the empirical proportion $X_i/t$ as a noisy surrogate for $Q_i$ and apply a standard density estimator under the assumed structure on $p$. Among several estimators of the underlying distribution, \citet{ye2021binomial} analyze a histogram estimator under a smooth density condition, and \citet{lee2025learning} study kernel density estimators under general H\"{o}lder smoothness and heterogeneous trials. 

Importantly, these results are upper bounds for particular estimators and do not characterize the fundamental difficulty of the problem. It remains unclear whether estimators based directly on empirical proportions are rate-optimal, or how the optimal rate depends jointly on sample size, assumed structure, and trials. More generally, to our knowledge, the minimax rates for estimating a smooth binomial mixing density remain unknown; most of the optimality theory for binomial mixtures has revolved around estimating the mixing \emph{distribution} under weak structural assumptions, with error measured in Wasserstein-1 distance \citep{tian2017learning, vinayak2019optimal, vinayak2019maximum}. While there are appeals to being structure-agnostic, the caveat is that the resulting rates are often slow, which may be particularly consequential in these measurement error problems where additional structure can substantially improve the attainable rates \citep{carroll1988optimal,fan1991optimal,fan1992deconvolution}.

To this end, we derive a lower bound on the minimax rate for pointwise estimation of an $s$-H\"{o}lder smooth mixing density in Section~\ref{sec:fundamental_limits}. Our derivation is based on Le Cam's two-point method and exploits the polynomial structure of the binomial mixture model, constructing smooth perturbations whose effects are either strongly attenuated by binomial sampling or completely invisible in the observed distribution. At a high level, our rates can be summarized as follows: in the ``small'' and ``intermediate'' $t$ regimes, binomial sampling limits the amount of detail about the mixing distribution that can be recovered, while smoothness determines how this loss of detail translates into pointwise density error. Once the number of trials is sufficiently large, the usual nonparametric smooth density estimation rate is recovered.

In Section~\ref{sec:attainability}, we show that the lower bound is tight by constructing a regularized orthogonal series estimator tailored to the binomial setting. We expand the mixing density in Legendre polynomials and exploit discrete Hahn polynomials to estimate its Legendre coefficients. Since the higher-order coefficients become increasingly unstable to recover, we regularize the resulting series; we call the resulting procedure the \emph{Legendre--Hahn series estimator}, and show that it attains the minimax rate.

Finally, in Section~\ref{sec:trial_heterogeneity}, we extend both the lower and upper bounds to heterogeneous trial parameters, which arise in a wide array of applications such as sports \citep{brown2008inseason}, ecology \citep{royle2004n,wenger2008estimating,dail2011models,wu2015bayesian}, and the social sciences \citep{kline2021reasonable,avivi2021adaptive,lee2025learning}. Our analysis shows that, under trial heterogeneity, the minimax rate depends on the full trial profile through the effective sample size for each polynomial degree.

\section{Setup and notation}\label{sec:setup}
We observe independent and identically distributed random variables $X_1,\ldots,X_n$ generated according to the model
\[
    X_i\mid Q_i=q\sim\operatorname{Bin}(t,q),
    \qquad
    Q_i\overset{\mathrm{iid}}{\sim}p,
    \qquad i=1,\ldots,n,
\]
where $p$ is an unknown density supported on $[0,1]$, and the common trial parameter $t\in\mathbb N$ is fixed and known. Throughout the main analysis, we focus on this homogeneous trial setting in which every observation has the same trial parameter $t$. In Section \ref{sec:trial_heterogeneity}, we extend our results to independent, non-identically distributed observations
\[
    X_i\mid Q_i=q\sim\operatorname{Bin}(t_i,q),
    \qquad
    Q_i\overset{\mathrm{iid}}{\sim}p, \qquad i=1,\ldots,n,
\]
where the trial parameters $t_1,\ldots,t_n$ are allowed to vary.

As alluded to, the key structure we adopt for the mixing density is H\"{o}lder smoothness. Let $s>0$ denote the smoothness parameter and write $s=r+\alpha$, where $r$ is the largest integer strictly smaller than $s$ and $0<\alpha\leq1$. In particular, when $s$ is an integer, we have $r=s-1$ and $\alpha=1$. For a function $f:[0,1]\to\R$, we define its $s$-H\"older norm as
\[
    \norm{f}_{C^s([0,1])} \equiv \sum_{j=0}^r\norm{f^{(j)}}_\infty + [f^{(r)}]_{C^\alpha([0,1])},
\]
with $[f^{(r)}]_{C^\alpha([0,1])}$ denoting the H\"older seminorm
\[
    [f]_{C^\alpha([0,1])} \equiv \sup_{\substack{q,q'\in[0,1]\\q\neq q'}}\frac{\abs{f(q)-f(q')}}{\abs{q-q'}^\alpha}.
\]
For constants $L,p_{\max}\in(0,\infty)$, the H\"{o}lder class of interest is
\[
    \Ht \equiv \left\{ p \geq 0: \int_0^1p(q)\dd q=1,\ \norm{p}_{C^s([0,1])}\leq L,\ \norm{p}_\infty\leq p_{\max}
    \right\}.
\]

Our goal is to estimate the density value $p(u)$ at a fixed interior point $u\in[\delta,1-\delta]$, where $\delta\in(0,1/2)$. More precisely, we study the minimax absolute pointwise risk
\[
    \cR_{n,t}
    \equiv
    \inf_{\widehat p}
    \sup_{p\in\Ht}
    \E_p\abs{\widehat p(u)-p(u)},
\]
where the infimum is over all estimators $\widehat p(u)$ measurable with respect to the observed sample $(X_1,\ldots,X_n)$, and $\E_p$ denotes expectation under the distribution induced by the mixing density $p$.

For real-valued functions $f$ and $g$ on $[0,1]$, we write
\[
    \ip{f}{g}_{L_2} \equiv \int_0^1 f(q)g(q)\dd q,
    \qquad
    \norm{f}_{L_2} \equiv \left\{\int_0^1f(q)^2\dd q\right\}^{1/2},
\]
and $\|f\|_\infty \equiv \sup_{u} |f(u)|$ for the sup norm. Further, for real-valued functions $f$ and $g$ on $\{0,\ldots,t\}$, we define the discrete inner product and the associated norm by
\[
    \langle f,g\rangle_{d,t} \equiv \frac{1}{t+1}\sum_{x=0}^{t}f(x)g(x), \qquad \|f\|_{d,t} \equiv \langle f,f\rangle_{d,t}^{1/2}.
\]

We use $\Pi_m$ to denote the space of algebraic polynomials on $[0,1]$ of degree at most $m$. For nonnegative quantities $a$ and $b$, we write $a\lesssim b$ if $a\leq Cb$ for some constant $C>0$, $a\gtrsim b$ if $a\geq Cb$, and $a\asymp b$ if both $a\lesssim b$ and $a\gtrsim b$. 
For probability measures $P_1,\ldots,P_n$, we write $\bigotimes_{i=1}^n P_i$ for their product measure. In particular, $P^{\otimes n}$ denotes the $n$-fold product of $P$.

\section{Fundamental limits}\label{sec:fundamental_limits}

We begin by providing a lower bound on the error of any estimator of the smooth binomial mixing density, which serves two related purposes. First, it provides a benchmark against which estimators can be evaluated. If an estimator attains the lower bound, then its rate cannot be improved uniformly over the model without imposing additional assumptions; conversely, a gap between a lower bound and the performance of an available estimator suggests that either a better estimator or a sharper analysis may be possible.

Second, a minimax lower bound describes the intrinsic difficulty of the statistical problem independently of any particular estimation procedure. In our setting, it clarifies the distinct roles of the sample size $n$ and the trial parameter $t$. In particular, $n$ determines how accurately the recoverable features of $p$ can be estimated; $t$ determines which features are recoverable in the first place.

Formally, our lower bound result can be stated as follows.

\begin{theorem} \label{thm:fundamental_lower_bound} Fix $s>0$, $\delta\in(0,1/2)$, and $u\in[\delta,1-\delta]$. Suppose that $L,p_{\max}>1$. Then, for all sufficiently large $n$, the minimax rate is lower bounded as
\[
    \cR_{n,t}
    \gtrsim
    \begin{cases}
        t^{-s},
        & t\leq\log n, \\[0.5em]
        \Big[t\left\{
            1+\log\left(\frac{n}{t^{s+1/2}}\right)
        \right\}\Big]^{-s/2},
        & \log n<t<n^{\frac{2}{2s+1}}, \\[0.5em]
        n^{-\frac{s}{2s+1}},
        & t\geq n^{\frac{2}{2s+1}}.
    \end{cases}
\]
\end{theorem}

\begin{figure}
    \centering\includegraphics[width=0.7\linewidth]{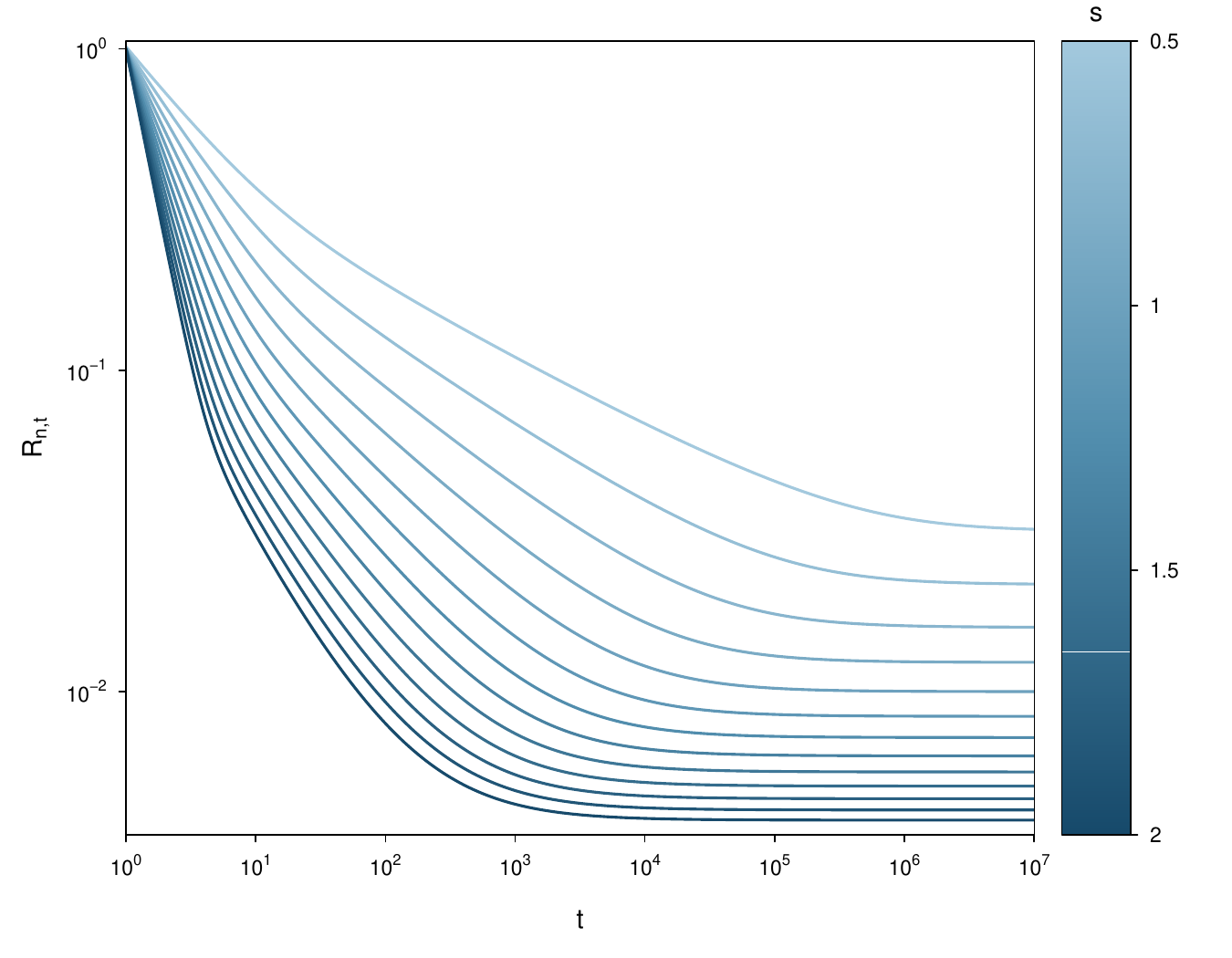}
    \caption{Smoothed illustration of the minimax rates as the trial parameter $t$ increases, with $n=10^6$ and different smoothness levels $s$. The curves decrease with $t$, but level off at order $n^{-s/(2s+1)}$ once $t$ reaches order $n^{2/(2s+1)}$. Higher smoothness gives a lower plateau reached with fewer trials.}
    \label{fig:minimax_rates}
\end{figure}
As can be seen, Theorem \ref{thm:fundamental_lower_bound} reveals three qualitatively different limitations. 

\begin{enumerate}
    \item \textbf{Small $t$ regime}: When $t \leq \log n$, the small number of trials is the primary bottleneck. In fact, there exist distinct smooth mixing densities that induce exactly the same distribution of the observed binomial sample while differing at $u$ by order $t^{-s}$. Thus, this error cannot be eliminated even with arbitrarily large sample size.
    \item \textbf{Intermediate $t$ regime}: When $\log n < t < n^{2/(2s+1)}$, some differences in the mixing density leave only a small imprint on the observed binomial distribution. A larger sample can detect increasingly smaller differences, so the achievable ``resolution'' now depends jointly on $t$ and $n$.
    \item \textbf{Large $t$ regime}: Finally, when $t \geq n^{2/(2s+1)}$, binomial sampling is sufficiently accurate for estimating an $s$-smooth density. Nevertheless, even direct observation of $Q_1,\ldots,Q_n$ could not yield a rate faster than $n^{-s/(2s+1)}$, which is the usual $1$-dimensional nonparametric density estimation rate.
\end{enumerate}

\begin{remark}
Existing minimax theory for binomial mixtures has primarily focused on estimation of the mixing distribution under Wasserstein-$1$ loss and weak structural assumptions \citep{tian2017learning,vinayak2019maximum}. This is a distinct problem from the one studied here; our target is the pointwise density value $p(u)$, our parameter space imposes $s$-H\"older smoothness, and our loss is the pointwise absolute error. Accordingly, we emphasize that their results should not be viewed as special cases of ours, nor vice versa.
\end{remark}

\begin{remark}
Nonetheless, it is worth comparing our intermediate $t$ rate with $(t\log n)^{-s/2}$, which closely resembles the minimax rate that \cite{vinayak2019maximum} obtain in their intermediate regime, $(t\log n)^{-1/2}$, but with smoothness incorporated. Note that
\begin{equation}\label{eq:lnt}
\frac{(t\ell_{n,t})^{-s/2}}{(t\log n)^{-s/2}}
=
\left(\frac{\log n}{\ell_{n,t}}\right)^{s/2}, \qquad \ell_{n,t} \equiv 1+\log\!\left(\frac{n}{t^{s+1/2}}\right).
\end{equation} 
In our intermediate regime, we have $1<\ell_{n,t}\leq 1+\log n$. Hence, if $\ell_{n,t}\gtrsim\log n$, then $\ell_{n,t}\asymp\log n$. 
In this case, therefore, the minimax rate is $(t\log n)^{-s/2}$ up to constants. 

If instead $\ell_{n,t}=o(\log n)$, the minimax rate is slower than $(t\log n)^{-s/2}$ by the diverging factor $(\log n/\ell_{n,t})^{s/2}$. In particular, when $t\asymp n^{2/(2s+1)}$ with $t<n^{2/(2s+1)}$, we have $\ell_{n,t}\asymp1$, and hence
\[
    (t\ell_{n,t})^{-s/2}
    \asymp t^{-s/2}
    \asymp n^{-s/(2s+1)}.
\]
Thus the intermediate rate is already of the same order as the
large $t$ rate in this constant-factor neighborhood of the
transition, whereas $(t\log n)^{-s/2}$ is too small by a factor
of order $(\log n)^{s/2}$.

\end{remark}

We derive the lower bound using Le Cam's two-point method \citep[Theorem 2.2]{tsybakov2008introduction}. The key idea is to construct two densities $p_+,p_-\in\Ht$ that are well separated at the target point $u$, but induce similar distributions for the observed binomial mixture sample. Let $q_+$ and $q_-$ denote the corresponding probability mass functions of a single observation $X$. If $|p_+(u)-p_-(u)|=2\Delta$ and the product distributions $q_+^{\otimes n}$ and $q_-^{\otimes n}$ are sufficiently close, then no estimator can reliably distinguish between $p_+$ and $p_-$. Subsequently, Le Cam's lemma gives a minimax lower bound of order $\Delta$.

We measure the distance between the observable laws using the $\chi^2$ divergence. Recall that the tensorization inequality \citep[Section 2.4]{tsybakov2008introduction} gives
\[
    \chi^2\left(q_+^{\otimes n},q_-^{\otimes n}\right)
    =
    \left\{1+\chi^2(q_+,q_-)\right\}^n-1
    \leq
    \exp\left\{n\chi^2(q_+,q_-)\right\}-1.
\]
Therefore, it is enough to construct $p_+$ and $p_-$ such that their pointwise separation is large, while $n\chi^2(q_+,q_-)$ remains bounded.

The argument has two parts. First, we construct two mixing densities whose difference involves only polynomial degrees of at most $t$ and is therefore, in principle, visible from the binomial observations. We choose a uniform reference density and perturb in a way that its effect on the observed probabilities is small enough that the two distributions remain difficult to distinguish from the observed sample. This yields what we call the ``intermediate $t$ lower bound,'' which depends jointly on $n$ and $t$.

Second, we construct two mixing densities whose difference is orthogonal to every polynomial of degree at most $t$. Since each binomial likelihood function is itself a polynomial of degree at most $t$, the two densities induce exactly the same distribution of the observed data. This yields an additional ``small $t$ lower bound'' that persists even with arbitrarily large sample size.

We now turn to the constructions used to prove Theorem \ref{thm:fundamental_lower_bound}. Section \ref{sec:shared_framework} introduces the perturbation framework shared by both lower bounds. We then apply this framework to the intermediate $t$ lower bound in Section \ref{sec:statistical_lower_bound}, where the two alternatives induce different but statistically close observed distributions, and to the small $t$ lower bound in Section \ref{sec:trial_only_lower_bound}, where the two alternatives induce exactly the same observed distribution. We conclude this section by presenting the combined lower bounds.

\subsection{Smooth perturbation framework}\label{sec:shared_framework}
We begin by describing a framework for constructing perturbed densities for the intermediate and small $t$ lower bounds. To reiterate, the main question is the following: how much can two mixing densities differ at the point $u$, while still producing nearly indistinguishable distributions of the observed binomials?

A common method for attaining lower bounds for pointwise density estimation problems is to perturb a reference density to obtain a null and an alternative that are separated at the target point, while inducing nearby distributions of the observed data; see,
e.g., \cite{fan1991optimal, roueff2005nonparametric, rebafka2015nonparametric,  belomestny2020nonparametric} for examples in related measurement error problems. The reference density and the perturbation are chosen so that the resulting null and alternative are valid densities that belong to the parameter class of interest.

We choose the uniform density $p_0(q) \equiv 1$ as the reference density, which is convenient in our setting. Under the assumption $L,p_{\max}>1$, sufficiently small perturbations of $p_0$ remain in $\Ht$. More importantly, it induces the uniform distribution on $\{0,\ldots,t\}$; to see this, let
\[
    b_{t,x}(q) \equiv \binom{t}{x}q^x(1-q)^{t-x}, \qquad x=0,\ldots,t,
\]
denote the Bernstein basis function and, for an integrable function $f$, define the binomial mixture operator $\cB_t$ as
\begin{equation}\label{eq:binomial_mixture}
    (\cB_tf)(x) \equiv \int_0^1b_{t,x}(q)f(q)\dd q, \qquad x=0,\ldots,t.    
\end{equation}
Then, the probability mass function induced by $p_0$ is
\begin{equation}\label{eq:discrete_uniform}
    q_0(x) \equiv (\cB_tp_0)(x) = \int_0^1b_{t,x}(q)\dd q = \frac{1}{t+1}.
\end{equation}

This considerably simplifies the $\chi^2$ divergence; if we thus choose our candidate densities to be of the form
\[
    p_{\pm,\Delta}(q) \equiv 1\pm\Delta g(q),
\]
where $g$ is the ``shape'' of a perturbation that satisfies
\begin{equation}\label{eq:g_constraints}
    \int_0^1g(q)\dd q=0, \qquad g(u)=1,
\end{equation}
and $\Delta>0$ denotes its ``amplitude'', then \eqref{eq:discrete_uniform} implies that the corresponding observed probability mass functions are
\[
    q_{\pm,\Delta}(x) \equiv (\cB_tp_{\pm,\Delta})(x) = \frac{1}{t+1} \pm \Delta(\cB_tg)(x).
\]
Therefore, if
\[
    q_{-,\Delta}(x) \geq c_0q_0(x)
    = \frac{c_0}{t+1},
    \qquad x=0,\ldots,t,
\]
for some fixed constant $c_0>0$, then the $\chi^2$ divergence between these observed laws satisfies
\[
    \chi^2(q_{+,\Delta},q_{-,\Delta})
    =4\Delta^2\sum_{x=0}^t\frac{(\cB_tg)(x)^2}{q_{-,\Delta}(x)}
    \leq \frac{4\Delta^2}{c_0}(t+1)
    \sum_{x=0}^t(\cB_tg)(x)^2.
\]
The first condition in \eqref{eq:g_constraints} ensures that $p_{\pm,\Delta}$ integrate to one, and the second condition yields $p_{+,\Delta}(u)-p_{-,\Delta}(u)=2\Delta$. 

Importantly, the quantity
\begin{equation}\label{eq:chi_squared_bound_squared}
    (t+1)\sum_{x=0}^t(\cB_tg)(x)^2,
\end{equation}
can be naturally expressed through the \emph{Bernstein--Durrmeyer operator}
\citep{durrmeyer1967formule,Derriennic1981,adell1995bernstein, CooperWaldron2000}. It is formally defined as
\begin{equation}\label{eq:bernstein_durrmeyer}
    (\cBD_tf)(q) \equiv (t+1)\sum_{x=0}^t b_{t,x}(q)(\cB_tf)(x).
\end{equation}
That is, $\cBD_t$ is obtained by first applying the binomial mixture operator \eqref{eq:binomial_mixture} and then mapping the resulting vector back to a polynomial on $[0,1]$. With $\cBD_t$, \eqref{eq:chi_squared_bound_squared} can simply be written in the quadratic form
\[
    (t+1)\sum_{x=0}^t(\cB_tg)(x)^2 = \ip{g}{\cBD_tg}_{L_2}.
\]
Hence, for a fixed pointwise separation, we would like to choose $g$ that obeys the constraints in \eqref{eq:g_constraints} and makes $\ip{g}{\cBD_tg}_{L_2}$ small. Once such a perturbation has been constructed, we can select the amplitude $\Delta$ to be as large as possible while keeping the perturbed densities in $\Ht$ and the corresponding product distributions difficult to distinguish.

\paragraph{The optimal unit perturbation.}
To choose an appropriate $g$, we begin by considering the optimal perturbation normalized so that $g(u)=1$, temporarily ignoring the zero-integral and smoothness constraints. Consider the Legendre polynomial of degree $k \geq 0$ on $[-1,1]$,
\[
    P_k(v)
    \equiv
    \frac{1}{2^kk!}
    \frac{\dd^k}{\dd v^k}(v^2-1)^k,
    \qquad -1\leq v\leq1,
\]
and define its normalized and shifted version on $[0,1]$ as
\[
    \phi_k(q) \equiv \sqrt{2k+1}P_k(2q-1).
\]
Our first lemma, which is Theorem III.3 of \cite{Derriennic1981}, shows that $\phi_k$ forms an eigenbasis of the Bernstein--Durrmeyer operator~\eqref{eq:bernstein_durrmeyer}.

\begin{lemma}[Theorem III.3 of \cite{Derriennic1981}]
\label{lem:bernstein_durrmeyer_eigen}
For $t\in\mathbb N$ and $k=0,\ldots,t$,
\[
    \cBD_t\phi_k = \lambda_{t,k}\phi_k, \qquad
    \lambda_{t,k}
    \equiv \prod_{j=0}^{k-1}\frac{t-j}{t+j+2},
    \qquad \lambda_{t,0} \equiv 1.
\]
\end{lemma}
Now, fix an integer $m\leq t$ and define the quantity
\[
    S_m(u) \equiv \sum_{k=0}^m\lambda_{t,k}^{-1}\phi_k(u)^2.
\]
If $g$ is a polynomial of degree at most $m$ of the form
\begin{equation}\label{eq:g}
    g(q) = \sum_{k=0}^m a_k\phi_k(q), \qquad g(u)=1,
\end{equation}
then Lemma \ref{lem:bernstein_durrmeyer_eigen} and the Cauchy--Schwarz inequality imply
\[
    \ip{g}{\cBD_tg}_{L_2}
    = \sum_{k=0}^m\lambda_{t,k}a_k^2
    \geq S_m(u)^{-1},
\]
with equality holding when $a_k \propto \lambda_{t,k}^{-1}\phi_k(u)$. It follows that
\[
    g^\star(q) \equiv \frac{\sum_{k=0}^m\lambda_{t,k}^{-1}\phi_k(u)\phi_k(q)}{S_m(u)},
\]
is the exact solution to
\[
    \min_{\substack{g\in\Pi_m\\g(u)=1}}\ip{g}{\cBD_tg}_{L_2},
\]
and the minimum value is $S_m(u)^{-1}$. Thus, $g^\star(q)$ is the optimal perturbation subject to $g(u)=1$, and $S_m(u)^{-1}$ is the smallest observable quadratic form required to create a unit perturbation at $u$ using polynomials of degree at most $m$.

Note that $g^\star$ need not integrate to zero, and that its derivatives can grow rapidly near the endpoints.\footnote{The derivatives of high-degree Legendre polynomials can grow rapidly near the endpoints; see equations (4.21.7) and (4.1.1) of \cite{szego75}.} Hence, a natural next step is to modify this optimal shape so that, after multiplication by a suitable amplitude, the resulting perturbed densities belong to $\Ht$, while keeping its effect on the observed distribution small.

\paragraph{Multiplication by a polynomial factor.}
To deal with the remaining constraints, we first introduce the polynomial factor
\[
    w_R(q) \equiv \left\{\frac{q(1-q)}{u(1-u)}\right\}^R,
    \qquad R\geq r+2.
\]
The denominator ensures that $w_R(u)=1$, while the numerator vanishes to order $R$ at both endpoints. Hence, multiplication by $w_R$ can help control the growth of the derivatives of high-degree Legendre polynomials near $0$ and $1$.

$w_R$ is also designed to help enforce the zero-integral constraint. Since $w_R$ is a polynomial of degree $2R$, for every $k>2R$, the orthogonality of the Legendre polynomials yields
\[
    \int_0^1 w_R(q)\phi_k(q)\dd q=0.
\]
Thus, multiplying a combination of Legendre polynomials of degrees greater than $2R$ by $w_R$ preserves its value at $u$ while producing a function that integrates to zero. 

Multiplying $g^\star$ by $w_R$, however, requires care for two reasons. First, $g^\star$ contains the terms of degrees $0,\ldots,2R$, for which the above orthogonality argument does not guarantee a zero integral. Second, we must also control the eigenvalues at nearby degrees, since multiplication by $w_R$ can shift each Legendre component by at most $2R$ degrees in either direction. 



For the intermediate $t$ lower bound, we obtain a zero integral by dropping the terms of degrees $0,\ldots,2R$, while keeping the inverse eigenvalue weighting over the remaining degrees. We also restrict $m$ so that the eigenvalues remain comparable over the degree shifts introduced by $w_R$. As shown in Section~\ref{sec:statistical_lower_bound}, removing these finitely many terms does not change the order of $S_m(u)$ over an admissible range of $m$, and multiplication by $w_R$ increases the quadratic form by at most a constant factor.

For the small $t$ lower bound, placing all terms above degree $t+2R$ handles both considerations directly. The perturbation then integrates to zero, and every degree resulting from multiplication by $w_R$ remains strictly larger than $t$. Since $\lambda_{t,k}=0$ for $k>t$, the quadratic form is exactly zero, and inverse eigenvalue weighting is unnecessary.

Having determined which degrees to use in each construction, the final remaining question is whether the normalization at $u$ remains well behaved. If only a single or a few degrees are retained, their values at $u$ may vanish or be arbitrarily close to zero; 
we therefore leverage sufficiently long consecutive blocks of the Legendre terms. For integers $a\geq0$ and $w\geq1$, define
\[
    \cJ(a,w)
    \equiv
    \{a+1,\ldots,a+w\}.
\]
The following lemma shows that the squared values of the Legendre polynomials cannot all be small over a sufficiently long consecutive block at an interior point.

\begin{lemma}
\label{lem:legendre_block}
Suppose $\delta\in(0,1/2)$. Then there exists an integer $W_0=W_{0,\delta}$ such that, uniformly over $u\in[\delta,1-\delta]$, $a\in\mathbb N_0$, and $w\geq W_0$,
\[
    \sum_{k\in\cJ(a,w)}\phi_k(u)^2 \asymp w.
\]
\end{lemma}

Accordingly, for the intermediate $t$ lower bound, we can restrict $m$ so that the block $\cJ(2R,m-2R)$ contains at least $W_0$ terms. For the small $t$ lower bound, we can adopt a block of width exactly $W_0$.
In what follows, we first apply this construction to the intermediate $t$ lower bound, and then give the corresponding construction for the small $t$ lower bound.

\subsection{Intermediate $t$ lower bound} \label{sec:statistical_lower_bound} 
We now concretely describe the perturbation for the intermediate $t$ lower bound, including the admissible range of the polynomial degrees and the choice of its amplitude. We begin by stating an assumption on the fixed constants used throughout the analysis. 

\begin{assumption}\label{assm:constants_homo} 
Let $R\geq r+2$ be the integer used in the definition of $w_R$, and let $W_0=W_{0,\delta}$ be as in Lemma \ref{lem:legendre_block}. Fix $\rho\in(0,1)$ and choose an integer $m_0\geq\max\{2W_0+2,4R\}$. Define 
\[ 
    t_0 \equiv \left\lceil \max\left\{ \frac{m_0}{\rho}, \frac{2R}{1-\rho} \right\} \right\rceil.
\]
We assume $t\geq t_0$.
\end{assumption} 

Next, we define the set of admissible polynomial degrees by 
\[ 
    \cM_t \equiv \left\{ m\in\mathbb N: m_0\leq m\leq\floor{\rho t} \right\},
\]
and, for each $m\in\cM_t$, consider
\[ 
    \cJ_{\intm} \equiv \cJ(2R,m-2R) = \{2R+1,\ldots,m\}. 
\] 
The choice of $m_0$ ensures $|\cJ_{\intm}| = m-2R \geq W_0$ so that Lemma \ref{lem:legendre_block} applies; further, the definition of $t_0$ gives $m+2R\leq t$, which guarantees that every degree appearing in the perturbation is at most $t$. 

Continuing our discussion from the previous section, we first verify that dropping the terms at $k=0,\ldots,2R$ does not change the order of $S_m(u)$.

\begin{lemma}
\label{lem:localized_block_comparison}
Suppose Assumption \ref{assm:constants_homo} holds. Then, for every $m\in\cM_t$, uniformly over $u\in[\delta,1-\delta]$,
\[
    \sum_{k\in\cJ_{\intm}}     \lambda_{t,k}^{-1}\phi_k(u)^2
    \asymp
    S_m(u).
\]
Consequently,
\[
    \min_{\substack{
        g\in\operatorname{span}
        \{\phi_{2R+1},\ldots,\phi_m\}\\
        g(u)=1
    }}
    \ip{g}{\cBD_tg}_{L_2}
    =
    \bigg\{\sum_{k\in\cJ_{\intm}}     \lambda_{t,k}^{-1}\phi_k(u)^2\bigg\}^{-1}
    \asymp
    S_m(u)^{-1}.
\]
\end{lemma}

For the proof, note that
\[
    S_m(u) - \sum_{k\in\cJ_{\intm}}
    \lambda_{t,k}^{-1}\phi_k(u)^2
    =
    \sum_{k=0}^{2R}
    \lambda_{t,k}^{-1}\phi_k(u)^2.
\]
Since the sum contains only a fixed number of terms, it is uniformly bounded over the admissible values of $t$ and $u$. On the other hand, $\lambda_{t,k}^{-1}\geq1$ and Lemma \ref{lem:legendre_block} give
\[
    \sum_{k\in\cJ_{\intm}}\lambda_{t,k}^{-1}\phi_k(u)^2
    \geq 
    \sum_{k\in\cJ_{\intm}}\phi_k(u)^2
    \gtrsim
    |\cJ_{\intm}|
    \geq
    W_0.
\]
Hence $\sum_{k\in\cJ_{\intm}}\lambda_{t,k}^{-1}\phi_k(u)^2 \asymp S_m(u)$. The minimization result then follows from the same Cauchy--Schwarz argument used to obtain $g^\star$, restricted to the degrees in $\cJ_{\intm}$.

We therefore define the perturbation for the intermediate $t$ lower bound as
\begin{equation}\label{eq:pert_stat}
    g_{\cJ_{\intm}}(q)
    \equiv
    w_R(q)
    \frac{\sum_{k\in\cJ_{\intm}}
    \lambda_{t,k}^{-1}\phi_k(u)\phi_k(q)}
         {\sum_{k\in\cJ_{\intm}}
    \lambda_{t,k}^{-1}\phi_k(u)^2},
\end{equation}
which, by construction, satisfies $g_{\cJ_{\intm}}(u)=1$ and $\int_0^1g_{\cJ_{\intm}}(q)\dd q=0$.

Our next lemma verifies that multiplication by $w_R$ preserves, up to constant factors, the quadratic-form bound obtained before the multiplication by $w_R$.

\begin{lemma}
\label{lem:operator_bound}
Suppose Assumption \ref{assm:constants_homo} holds. Then, for every $m\in\cM_t$, uniformly over $u\in[\delta,1-\delta]$, 
\[
    \ip{g_{\cJ_{\intm}}}
       {\cBD_tg_{\cJ_{\intm}}}_{L_2}
    \lesssim
    S_m(u)^{-1}.
\]
\end{lemma}

The intuition is as follows: since $w_R$ has degree $2R$, multiplying $\phi_k$ by $w_R$ can only produce degrees $\ell$ satisfying $|\ell-k|\leq 2R$. If $\ell\geq k$, monotonicity gives $\lambda_{t,\ell}\leq\lambda_{t,k}$; if instead 
$\ell<k$, the restriction $k\leq m\leq \lfloor\rho t\rfloor$ yields
\[
    \lambda_{t,\ell}\lesssim\lambda_{t,k},
    \qquad
    |\ell-k|\leq2R.
\]
Thus, every degree $\ell$ appearing in $w_R\phi_k$ satisfies $\lambda_{t,\ell}\lambda_{t,k}^{-2} \lesssim \lambda_{t,k}^{-1}$. Since only a fixed number of terms contribute to each degree $\ell$, multiplication by $w_R$ increases the quadratic form by at most a constant factor.

\paragraph{Choice of amplitude.}
Because $g_{\cJ_{\intm}}(u)=1$, $\Delta$ is exactly half of the pointwise separation between the two alternatives. We choose $\Delta$ as large as our bounds allow, while ensuring that the perturbed densities remain in the H\"{o}lder class and that the corresponding observed distributions remain difficult to distinguish. Lemma \ref{lem:operator_bound} encodes the latter constraint, as it implies that the $\chi^2$ divergence is bounded by a constant multiple of $\Delta^2/S_m(u)$. Hence, to keep the $n$-sample divergence bounded, it suffices to take
\[
\Delta \leq c\sqrt{\frac{S_m(u)}{n}},
\]
for a sufficiently small $c>0$.

We introduce one final lemma to verify the first constraint.

\begin{lemma}
\label{lem:holder_bound}
Suppose Assumption \ref{assm:constants_homo} holds. Then, for every $m\in\cM_t$, uniformly over $u\in[\delta,1-\delta]$,
\[
    \norm{g_{\cJ_{\intm}}}_{C^s([0,1])}
    \lesssim
    m^s.
\]
\end{lemma}


The proof first establishes that
\[
    \|g_{\cJ_{\intm}}^{(j)}\|_\infty \lesssim m^j, \qquad j=0,\ldots,r+1, 
\]
using the fact that the vanishing of $w_R$ at the endpoints compensates for the growth of the Legendre derivatives. Interpolating between the bounds for the $r$-th and $(r+1)$-st derivatives then gives
\[
    [g_{\cJ_{\intm}}^{(r)}]_{C^\alpha([0,1])} \lesssim m^{r+\alpha} = m^s.
\]


We therefore choose $\Delta$ to satisfy both constraints and define the perturbed densities for the intermediate $t$ lower bound as follows.

\begin{definition}[Densities for the intermediate $t$ lower bound]
\label{def:statistical_lower_bound_densities}
For $q\in[0,1]$ and a sufficiently small constant $c>0$, the perturbed densities for the intermediate $t$ lower bound are
\[
    p_{\pm,\Delta_m}(q)
    \equiv
    1\pm\Delta_mg_{\cJ_{\intm}}(q),
    \qquad
    \Delta_m
    \equiv
    c\min\left\{
        m^{-s},
        \sqrt{\frac{S_m(u)}{n}}
    \right\}.
\]
\end{definition}

\begin{figure}[t!]
\centering
\includegraphics[width=\textwidth]{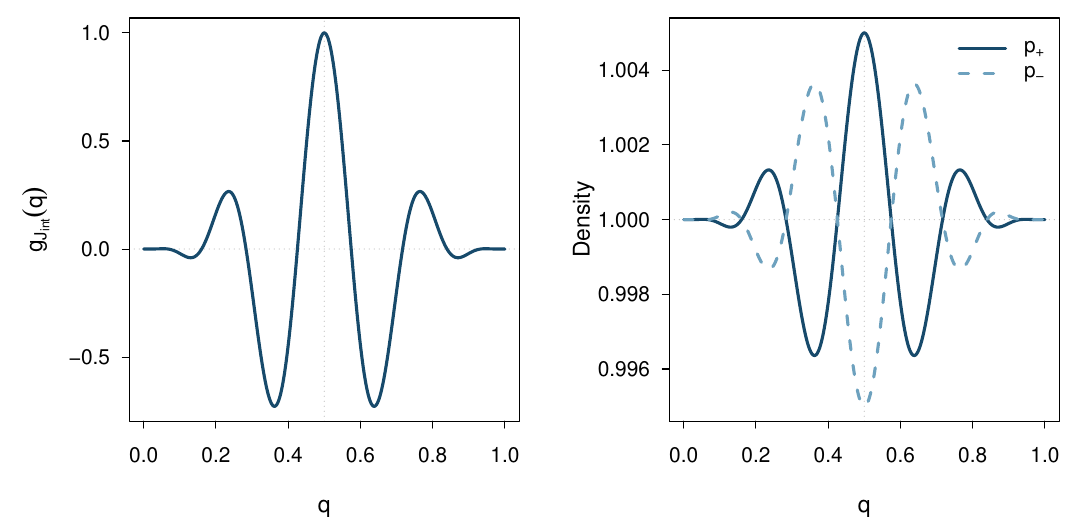}
\caption{Example shape $g_{\cJ_{\intm}}(q)$ and the corresponding perturbed densities $p_{\pm,\Delta_m}(q) =1\pm\Delta_mg_{\cJ_{\intm}}(q)$, with target point $u=0.5$. The perturbation is normalized and integrates to zero, so the separation created at the target point is accompanied by deviations of opposite sign elsewhere on the interval. $w_R$ suppresses the perturbation near the endpoints.} 
\end{figure}

By Lemma \ref{lem:holder_bound}, the first term in the definition of $\Delta_m$ guarantees
\[
    \Delta_m \norm{g_{\cJ_{\intm}}}_{C^s}\lesssim c,
\]
so $p_{\pm,\Delta_m}\in\Ht$ for sufficiently small $c$. The same choice of $c$ also yields $p_{-,\Delta_m}\geq1/2$, and hence
\[
    q_{-,\Delta_m}(x)
    \geq
    \frac{1}{2(t+1)},
    \qquad
    x=0,\ldots,t.
\]
The second term in the definition of $\Delta_m$, together with Lemma
\ref{lem:operator_bound}, gives
\[
    n\chi^2(q_{+,\Delta_m},q_{-,\Delta_m})
    \lesssim
    \frac{n\Delta_m^2}{S_m(u)}
    \lesssim
    c^2.
\]
The tensorization inequality then yields
\[
    \chi^2\left(
        q_{+,\Delta_m}^{\otimes n},
        q_{-,\Delta_m}^{\otimes n}
    \right)
    \leq
    \exp(Cc^2)-1.
\]
Thus, the corresponding observed product distributions remain difficult to distinguish, while the mixing densities are separated at $u$ by $2\Delta_m$. The following theorem gives the resulting intermediate $t$ lower bound.

\begin{theorem}[Intermediate $t$ lower bound] \label{thm:stat_lower_bound} Fix $u\in[\delta,1-\delta]$. Suppose that $L,p_{\max}>1$ and that Assumption \ref{assm:constants_homo} holds. Then 
\[ 
    \cR_{n,t} \gtrsim \max_{m\in\cM_t} \min\left\{ m^{-s}, \sqrt{\frac{S_m(u)}{n}} \right\}. 
\] 
\end{theorem}

This lower bound can be interpreted as follows. At each fixed $m$, the inner minimum limits the amplitude so that both the smoothness and indistinguishability constraints are satisfied. The outer maximum then selects an admissible $m$ that yields the largest feasible pointwise separation. Combining this result with the small $t$ lower bound gives a more familiar bias--variance form (Section~\ref{sec:combined_lower_bound}).

\subsection{Small $t$ lower bound}
\label{sec:trial_only_lower_bound}

The intermediate $t$ lower bound concerns perturbations that affect the observed binomial probabilities but whose effects are too small to distinguish reliably from a finite sample. Small $t$ lower bound, on the other hand, regards the fundamental identifiability barrier of binomial mixtures. 

Since each Bernstein basis function $b_{t,x}$ is a polynomial of degree at most $t$, any perturbation orthogonal to $\Pi_t$ does not affect the observed distribution. The small $t$ lower bound is constructed precisely by exploiting this; we leverage a block above degree $t+2R$ defined as
\[
    \cJ_{\sml}
    \equiv
    \cJ(t+2R,W_0)
    =
    \{t+2R+1,\ldots,t+2R+W_0\},
\]
and choose our perturbation to be
\[
    g_{\cJ_{\sml}}(q)
    \equiv w_R(q) \frac{\sum_{k\in\cJ_{\sml}}\phi_k(u)\phi_k(q)}{\sum_{k\in\cJ_{\sml}}\phi_k(u)^2},
\]
which is now unweighted since the inverse eigenvalues are neither defined nor necessary. 
Since the block has width $W_0$, the denominator in $g_{\cJ_{\sml}}(q)$ is again bounded away from zero uniformly over $u\in[\delta,1-\delta]$ by Lemma \ref{lem:legendre_block}. Further, as with $g_{\cJ_{\intm}}$, we have by construction that $g_{\cJ_{\sml}}(u)=1$ and $\int_0^1g_{\cJ_{\sml}}(q)\dd q=0$. The same argument as in Lemma \ref{lem:holder_bound} yields the following bound.

\begin{lemma}
\label{lem:trial_perturbation}
Suppose that $R\geq r+2$. Then, uniformly over $u\in[\delta,1-\delta]$ and $t\in\mathbb N$, 
\[
    \norm{g_{\cJ_{\sml}}}_{C^s([0,1])}
    \lesssim
    t^s.
\]
\end{lemma}

We therefore choose an amplitude of order $t^{-s}$, which satisfies the H\"{o}lder constraint for a sufficiently small multiplicative constant.

\begin{definition}[Densities for the small $t$ lower bound]
For a sufficiently small constant $c>0$, the perturbed densities for the small $t$ lower bound are
\[
    p_{\pm,\Delta_t}(q)
    \equiv
    1\pm\Delta_tg_{\cJ_{\sml}}(q),
    \qquad
    \Delta_t
    \equiv
    ct^{-s}.
\]
\end{definition}

By Lemma \ref{lem:trial_perturbation}, the two densities belong to $\Ht$ for sufficiently small $c$ and differ at the target point by $2ct^{-s}$. At the same time, since $\cB_tg_{\cJ_{\sml}}=0$, 
\[
    q_{\pm,\Delta_t}(x)
    =
    \frac{1}{t+1},
    \qquad
    x=0,\ldots,t,
\]
so the two mixing densities induce exactly the same distribution of the observed sample. We thus have the following small $t$ lower bound.

\begin{theorem}[Small $t$ lower bound]
\label{thm:trial_lower_bound}
Fix $u\in[\delta,1-\delta]$ and suppose that $L,p_{\max}>1$. Then, for every $t\in\mathbb N$, 
\[
    \cR_{n,t}
    \gtrsim
    t^{-s}.
\]
\end{theorem}

\subsection{Combined lower bound}
\label{sec:combined_lower_bound}

Combining the intermediate $t$ and small $t$ lower bounds in Theorems~\ref{thm:stat_lower_bound}--\ref{thm:trial_lower_bound} yields
\[
    \cR_{n,t} \gtrsim \max\left[t^{-s},
        \ \max_{m_0\leq m\leq\floor{\rho t}}
        \min\left\{m^{-s},\sqrt{\frac{S_m(u)}{n}}\right\}\right].
\]
As alluded to, we can write the combined lower bound in a more familiar bias--variance form. 

\begin{corollary}
\label{cor:lower_bound_bias_variance}
Fix $u\in[\delta,1-\delta]$. Suppose that $L,p_{\max}>1$ and that Assumption \ref{assm:constants_homo} holds. Then
\[
    \cR_{n,t} \gtrsim \min_{m\in\cM_t}\left\{m^{-s} + \sqrt{\frac{S_m(u)}{n}}\right\}.
\]
\end{corollary}

To see this, note that $m^{-s}$ decreases with $m$, while $\sqrt{S_m(u)/n}$ increases with $m$. If their first crossing occurs above $m_0$, the sum at the previous degree is bounded by a constant multiple of the inner minimum at the crossing. If the variance term remains smaller than the bias throughout the admissible range, the sum at $m=\lfloor\rho t\rfloor$ is of order $t^{-s}$, which is covered by the small $t$ lower bound. Finally, if the variance term is already at least as large as the bias at $m_0$, the inner minimum there equals $m_0^{-s}$; the uniform boundedness of $S_{m_0}(u)$ bounds the sum by a constant multiple of this value.

\begin{remark}
The three lower bounds in
Theorem~\ref{thm:fundamental_lower_bound} follow from
Corollary~\ref{cor:lower_bound_bias_variance} and the bound
\[
    S_m(u)\gtrsim
    m\exp\left(c\frac{m^2}{t}\right),
\]
where $c>0$ depends only on $\rho$, uniformly over $m\in\cM_t$ and $u\in[\delta,1-\delta]$.

\begin{enumerate}
    \item \textbf{Small $t$ regime}: every $m \in \cM_t$ satisfies $m\leq\rho t$, so $m^{-s}\gtrsim t^{-s}$. Corollary~\ref{cor:lower_bound_bias_variance} therefore gives
    $\cR_{n,t}\gtrsim t^{-s}$.

    \item \textbf{Intermediate $t$ regime}: for a sufficiently large constant $C>0$,  $m\leq C\sqrt{t\ell_{n,t}}$ and $m>C\sqrt{t\ell_{n,t}}$ satisfy
    \[
        m^{-s}\gtrsim(t\ell_{n,t})^{-s/2},
        \qquad
        \sqrt{\frac{S_m(u)}{n}}
        \gtrsim(t\ell_{n,t})^{-s/2},
    \]
    respectively, where the latter follows from the lower bound on $S_m(u)$. This yields $\cR_{n,t}\gtrsim(t\ell_{n,t})^{-s/2}$.

    \item \textbf{Large $t$ regime}: using $S_m(u)\gtrsim m$ and minimizing over all $m>0$ gives
    \[
        \cR_{n,t}
        \gtrsim
        \min_{m>0}\left\{m^{-s}+\sqrt{\frac{m}{n}}\right\}
        \asymp n^{-s/(2s+1)}.
    \]
\end{enumerate}

More details can be found in the proof of Corollary~\ref{cor:simplified_minimax_rate}.
\end{remark}

\section{Attainability}\label{sec:attainability}
In this section, we show that the lower bounds from Section
\ref{sec:fundamental_limits} are attainable by a suitably regularized orthogonal series estimator. The key additional fact we leverage in constructing this estimator is that each Legendre coefficient of the mixing density can be estimated unbiasedly from an associated polynomial of the observed binomial sample. These observable polynomials turn out to be the \emph{Hahn polynomials}, which we introduce below.

\subsection{Proposed estimator}\label{sec:proposed_estimator}
For $k=0,\ldots,t$, the normalized Hahn polynomials on $\{0,\ldots,t\}$ are formally defined as
\[
    \psi_{k,t}(x)
    \equiv
    \lambda_{t,k}^{-1/2}(t+1)(\cB_t\phi_k)(x),
    \qquad
    x=0,\ldots,t.\footnote{These polynomials are known as discrete Chebyshev or Gram polynomials \citep{ciesielski1987explicit,waldron2006bernstein}.}
\]
Interestingly, these polynomials arise naturally from Lemma \ref{lem:bernstein_durrmeyer_eigen}. Note that, for $k,\ell=0,\ldots,t$, we have
\[
    \ip{(t+1)\cB_t\phi_k}{(t+1)\cB_t\phi_\ell}_{d,t} = \ip{\phi_k}{\cBD_t\phi_\ell}_{L_2} = \lambda_{t,\ell}\mathbbm 1(k=\ell). 
\]
Thus, applying the scaled binomial mixture operator $(t+1)\cB_t$ to the
shifted Legendre polynomials produces mutually orthogonal functions on
$\{0,\ldots,t\}$, with
\[
    \norm{(t+1)\cB_t\phi_k}_{d,t}^2 = \lambda_{t,k}.
\]
Dividing each resulting function by its norm therefore yields the polynomials.
Thus, after normalization, the function obtained by applying the binomial mixture operator to the $k$-th shifted Legendre polynomial is the $k$-th Hahn polynomial, where $\sqrt{\lambda_{t,k}}$ quantifies how much the $k$-th Legendre component is attenuated. Our next proposition precisely characterizes the corresponding relationship from the Hahn polynomials back to the Legendre polynomials: conditional averaging of $\psi_{k,t}(X)$ given $Q=q$ recovers $\phi_k(q)$ with $\sqrt{\lambda_{t,k}}$.

\begin{proposition}
\label{prop:singular_system}
For $k=0,\ldots,t$, we have
\[
    \E_p\{\psi_{k,t}(X)\mid Q=q\} = \sqrt{\lambda_{t,k}}\phi_k(q).
\]
\end{proposition}

Proposition \ref{prop:singular_system} implies that, if
$\theta_k\equiv\ip{p}{\phi_k}_{L_2}$ is the $k$-th Legendre coefficient of $p$, then
\begin{equation}\label{eq:coefficient_identity}
    \E_p\{\psi_{k,t}(X)\} = \sqrt{\lambda_{t,k}}\theta_k.
\end{equation}
Thus, for every $k=0,\ldots,t$, an unbiased estimator of $\theta_k$ can be constructed simply as
\begin{equation}\label{eq:coefficient_estimator}
    \widehat\theta_k
    \equiv
    \frac{1}{n\sqrt{\lambda_{t,k}}}
    \sum_{i=1}^n\psi_{k,t}(X_i).
\end{equation}
If $Q_1,\ldots,Q_n$ were directly observed, $\theta_k$ could be
estimated simply by averaging $\phi_k(Q_i)$. 
\eqref{eq:coefficient_identity} reveals a crucial structural fact that, under binomial mixture observations, the coefficient is instead observed after being \emph{multiplied by} $\sqrt{\lambda_{t,k}}$. Hence, dividing by $\sqrt{\lambda_{t,k}}$ gives rise to the natural unbiased estimator in \eqref{eq:coefficient_estimator}. 

At the same time, as $\lambda_{t,k}$ decreases with $k$, this division also amplifies variance. Thus, while the lower-degree coefficients can be estimated relatively stably, directly
inverting increasingly high degrees becomes progressively less stable. A simple, intuitive solution is to regularize the unstable components; we thus estimate $p$ using a regularized orthogonal series estimator, which we call the \emph{Legendre--Hahn series
estimator}.\footnote{This is a usual aspect of a statistical inverse problem where inversion is carried out along the singular functions while the less stable components are regularized. For reference, see Sections 1.2.2.1--1.2.2.2 of \cite{cavalier2011inverse}.}

\begin{definition}[Legendre--Hahn series estimator]
\label{def:legendre_hahn_estimator}
Fix a function $\eta\in C^\infty([0,\infty))$ satisfying
\[
    0\leq\eta\leq1,
    \qquad
    \eta(v)=1\ \text{for }v\in[0,1],
    \qquad
    \eta(v)=0\ \text{for }v\in[2,\infty).
\]
For an integer $1\leq m\leq\floor{t/2}$, let $\widehat\theta_k$ be defined as in \eqref{eq:coefficient_estimator}, and define
\begin{equation}\label{eq:legendre_hahn_estimator}
    \widehat p_m(q)
    \equiv
    \sum_{k=0}^{2m}
    \eta(k/m)\widehat\theta_k\phi_k(q),
    \qquad
    q\in[0,1].
\end{equation}
\end{definition}

As can be seen, $\eta$ is the regularization function that preserves all polynomial degrees $k \leq m$ (``stable components''), discards those with $k \geq 2m$ (``unstable components''), and smooths the intermediate. This smoothing is to obtain a uniformly stable approximation operator under the sup norm, which allows us to control the bias. $2m \leq t$ ensures that every coefficient appearing in \eqref{eq:legendre_hahn_estimator} is identifiable. 

\subsection{Bias and variance}\label{sec:upper_bias_variance}
We now analyze the bias and the variance of our proposed estimator. Although $\cR_{n,t}$ is defined with respect to the absolute pointwise risk, it is convenient to first control the pointwise root mean squared error. The corresponding upper bound under absolute risk then follows from the Cauchy--Schwarz inequality. 


\begin{theorem}[Upper bound for Legendre--Hahn series estimator]
\label{thm:upper_bound}
Fix $\delta\in(0,1/2)$ and $u\in[\delta,1-\delta]$. Suppose that $p\in\Ht$ and $1\leq m\leq\floor{t/2}$. Then the estimator $\widehat p_m$ in Definition \ref{def:legendre_hahn_estimator} satisfies
\[
    \left[\E_p\{\widehat p_m(u)-p(u)\}^2\right]^{1/2} 
    \lesssim m^{-s} + \sqrt{\frac{S_{2m}(u)}{n}}.
\]
\end{theorem}

The proof is simple. First, \eqref{eq:coefficient_identity} implies
\[
    \E_p\{\widehat p_m(q)\} = (V_mp)(q), \qquad (V_mf)(q) \equiv \sum_{k=0}^{2m} \eta(k/m) \ip{f}{\phi_k}_{L_2}\phi_k(q).
\]
Hence, we can decompose the estimation error at $u$ as
\[
    \widehat p_m(u)-p(u) = \{\widehat p_m(u)-(V_mp)(u)\} + \{(V_mp)(u)-p(u)\}.
\]

Now, we note that $V_mf$ is uniformly bounded.

\begin{lemma}\label{lem:filtered_legendre}
Let $\eta\in C^\infty([0,\infty))$ satisfy 
\[ 
    0\leq\eta\leq1, \qquad \eta(v)=1 \text{ for } v\in[0,1], \qquad \eta(v)=0 \text{ for } v\in[2,\infty).
\] 
Then
\[
    \sup_{m\geq1}\sup_{\norm{f}_\infty\leq1}\norm{V_mf}_\infty < \infty.
\]
\end{lemma}

Since $\eta(k/m)=1$ for $k\leq m$, $V_m$ reproduces every polynomial in $\Pi_m$. Combining this reproduction property with Lemma \ref{lem:filtered_legendre} yields
\[
    \norm{V_mp-p}_\infty
    \lesssim
    \inf_{\pi\in\Pi_m}\norm{p-\pi}_\infty.
\]
Then, standard polynomial approximation theory for $C^s$ functions implies
\[
    \norm{V_mp-p}_\infty \lesssim m^{-s}.
\]

For the variance term, note that the observed probability mass function satisfies
\[
    \Pp_p(X=x)
    =
    \int_0^1b_{t,x}(q)p(q)\dd q
    \leq
    \frac{p_{\max}}{t+1}, \qquad x=0,\ldots,t.
\]
Thus, together with the discrete orthonormality of the Hahn polynomials, we have
\begin{align*}
    \Var_p\{\widehat p_m(u)\}
    &\leq
    \frac{1}{n}\E_p\bigg[\bigg\{\sum_{k=0}^{2m}
    \eta(k/m)
    \frac{\phi_k(u)}{\sqrt{\lambda_{t,k}}}
    \psi_{k,t}(X)\bigg\}^2\bigg]\\
    &\leq
    \frac{p_{\max}}{n(t+1)}
    \sum_{x=0}^t \bigg\{\sum_{k=0}^{2m}
    \eta(k/m)
    \frac{\phi_k(u)}{\sqrt{\lambda_{t,k}}}
    \psi_{k,t}(x)\bigg\}^2\\
    &=
    \frac{p_{\max}}{n}
    \sum_{k=0}^{2m}
    \eta(k/m)^2
    \lambda_{t,k}^{-1}\phi_k(u)^2\\
    &\leq
    \frac{p_{\max}S_{2m}(u)}{n}.
\end{align*}
Combining the bias and variance bounds proves Theorem \ref{thm:upper_bound}. Since $\E|Z|\leq(\E Z^2)^{1/2}$, Theorem \ref{thm:upper_bound} also gives the same upper bound under the absolute pointwise risk defining $\cR_{n,t}$.

\subsection{Overall minimax rate}\label{sec:overall_minimax_rate}

Combining the upper and lower bounds shows that  the bias--variance trade-off achieved by our orthogonal series estimator is minimax optimal.

\begin{corollary}[Minimax rate]\label{cor:minimax_rate}
Fix $u\in[\delta,1-\delta]$. Suppose that Assumption \ref{assm:constants_homo} holds and $L,p_{\max}>1$. Then
\[
    \cR_{n,t}
    \asymp
    \min_{m\in\cM_t}
    \left\{
        m^{-s}
        +
        \sqrt{\frac{S_m(u)}{n}}
    \right\}.
\]
\end{corollary}

The lower bound follows from Section \ref{sec:combined_lower_bound}. For the upper bound,
we can apply Theorem~\ref{thm:upper_bound} with degree $\floor{m/2}$, use the monotonicity of $S_m(u)$, and then minimize over $m_0\leq m\leq\floor{\rho t}$.

This minimum can be thought of as selecting the polynomial degree that optimally balances bias and variance. 
Since $s$ is often unknown in practice, this motivates data-driven selection of the polynomial degree; one possible approach is Lepski's method
\citep{lepski1991problem}, which begins at the smallest candidate degree and increases it until the estimate at $u$ agrees with those at all larger degrees within certain comparison thresholds. 
These thresholds must allow for differences caused by sampling noise; for $j>m$, the variance calculation above yields
\[
    \operatorname{Var}_p
    \{\widehat p_j(u)-\widehat p_m(u)\}
    \leq \frac{p_{\max}S_{2j}(u)}{n}.
\]
Its square root, therefore, bounds the standard deviation of each difference without requiring knowledge of $s$. We leave choosing the thresholds 
and establishing the resulting adaptive pointwise risk bounds for future work. 


Our last corollary for this section derives the rates shown earlier in Theorem \ref{thm:fundamental_lower_bound}, which are more interpretable. The lower bounds follow
as outlined in Remark~3; the matching upper bounds follow from Corollary~\ref{cor:minimax_rate} by using the bound $S_m(u) \lesssim m\exp(Cm^2/t)$
and choosing the polynomial degree appropriately in each regime.

\begin{corollary}
\label{cor:simplified_minimax_rate}
Fix $u\in[\delta,1-\delta]$. Suppose that Assumption \ref{assm:constants_homo} holds and  $L,p_{\max}>1$. Then, for all sufficiently large $n$,
\[
    \cR_{n,t}
    \asymp
    \begin{cases}
        t^{-s},
        & t \leq \log n, \\[0.5em]
        \Big[t\left\{ 1+\log\left(\frac{n}{t^{s+1/2}}\right) \right\} \Big]^{-s/2},
        & \log n <  t < n^{\frac{2}{2s+1}}, \\[0.5em]
        n^{-\frac{s}{2s+1}},
        & t \geq n^{\frac{2}{2s+1}}.
    \end{cases}
\]
\end{corollary}

\section{Extension to trial heterogeneity}
\label{sec:trial_heterogeneity}
We now extend the lower and upper bounds to heterogeneous trials. We focus on describing how the heterogeneous trial parameters affect the information available for estimation and, in turn, the minimax rate. Construction of the lower bounds adopts the smooth perturbation framework in Section~\ref{sec:shared_framework} with appropriate modifications
to accommodate trial heterogeneity, so we defer its details to  Appendix~\ref{sec:trial_heterogeneity_details}.

Under trial heterogeneity, we observe independent random variables
\[
    X_i\mid Q_i=q\sim\operatorname{Bin}(t_i,q),
    \qquad
    Q_i\overset{\mathrm{iid}}{\sim}p,
    \qquad i=1,\ldots,n,
\]
where the trial parameters $\mathbf t\equiv(t_1,\ldots,t_n)$ are again assumed fixed and known. We denote the corresponding absolute pointwise minimax risk by
\[
    \cR_{n,\mathbf t}
    \equiv
    \inf_{\widehat p}\sup_{p\in\Ht}
    \E_p\abs{\widehat p(u)-p(u)},
\]
and let $T\equiv\max_{1\leq i\leq n}t_i$ denote the largest trial parameter.

At a high level, the main difference from the homogeneous setting is that observations with different trial counts can provide different amounts of information at a given polynomial degree. To make this precise, for $k\geq0$, we define
\begin{equation}
\label{eq:aggregate_information}
    \cI_k \equiv \sum_{i=1}^n\lambda_{t_i,k} = \sum_{i:t_i\geq k}\lambda_{t_i,k}, \qquad \cI_0=n.
\end{equation}
As we show below, $\lambda_{t_i,k}$ quantifies the contribution of observation $i$ at degree $k$, so $\cI_k$ makes precise the sense in which ``different observations need not provide the same amount of information at a given polynomial degree''. In particular, $\cI_k$ can be interpreted as the \emph{effective sample size} for estimating the Legendre coefficient $\theta_k=\E_p\{\phi_k(Q)\}$, or more simply put, the \emph{degree-specific effective sample size}.

To see this, it is helpful to first discuss how an optimal estimator for $\theta_k$ under heterogeneous trials should be constructed. Recall that Proposition~\ref{prop:singular_system} 
engendered the natural unbiased estimator of the Legendre coefficient \eqref{eq:coefficient_estimator} under trial homogeneity. In the current setting, Proposition~\ref{prop:singular_system} now gives
\[
    \E_p\left\{\frac{\psi_{k,t_i}(X_i)}{\sqrt{\lambda_{t_i,k}}}\right\} = \theta_k,
\]
for every observation with $t_i\geq k$; thus, every observation with at least $k$ trials provides an unbiased estimate of $\theta_k$ (based on that observation alone). Since each estimate is independent, we can consider an estimator based on weighted averages of the form
\[
    \widehat\theta_k = \sum_{i:t_i\geq k}w_{i,k}\frac{\psi_{k,t_i}(X_i)}{\sqrt{\lambda_{t_i,k}}},
    \qquad \sum_{i:t_i\geq k} w_{i,k}=1.
\]
Since any such choice of weights preserves unbiasedness, it remains to choose the weights to account for the different variabilities of the individual estimators.

Suppose for simplicity that the mixing density $p$ is uniform. In this case, $X_i$ is uniform on $\{0,\ldots,t_i\}$ and, for $k\geq1$,
\[
    \E_p\{\psi_{k,t_i}(X_i)\} = \sqrt{\lambda_{t_i,k}} \ip{1}{\phi_k}_{L_2} = 0.
\]
Therefore, by the discrete orthonormality of the Hahn polynomials, we
have
\[
    \Var_p\left\{\frac{\psi_{k,t_i}(X_i)}{\sqrt{\lambda_{t_i,k}}}\right\}
    = \frac{1}{\lambda_{t_i,k}}
    \E_p\{\psi_{k,t_i}(X_i)^2\}
    = \frac{1}{\lambda_{t_i,k}(t_i+1)}
    \sum_{x=0}^{t_i}\psi_{k,t_i}(x)^2
    = \lambda_{t_i,k}^{-1}.
\]
Hence the natural inverse-variance weight of observation $i$ is $\lambda_{t_i,k}$ under the uniform mixing density, which implies that we should set 
\[
    w_{i,k} = \frac{\lambda_{t_i,k}}{\sum_{j:t_j\geq k}\lambda_{t_j,k}}
    = \frac{\lambda_{t_i,k}}{\cI_k},
\]
so that
\begin{equation}
\label{eq:legendre_coeff_estimator_het}
    \widehat\theta_k \equiv \frac{1}{\cI_k}\sum_{i:t_i\geq k}\sqrt{\lambda_{t_i,k}}\psi_{k,t_i}(X_i), \qquad 0\leq k\leq T.
\end{equation}
\eqref{eq:legendre_coeff_estimator_het} reduces exactly to
\eqref{eq:coefficient_estimator} in the homogeneous case. For $k\geq1$, its variance is
\[
    \Var_p(\widehat\theta_k) = \frac{1}{\cI_k^2}\sum_{i:t_i\geq k}\lambda_{t_i,k} = \frac{1}{\cI_k}.
\]
On the other hand, $\Var_p\{\phi_k(Q)\}=1$ if $Q_1,\ldots,Q_n$ were observed directly, so their sample average would have variance $1/n$. Thus, relative to directly observing the latent probabilities, the heterogeneous binomial sample has effective sample size $\cI_k$ for estimating the $k$th Legendre coefficient. This is the specific sense in which $\cI_k$ is a \emph{degree-specific effective sample size}.

Importantly, the same choice of weights remains well controlled uniformly over our parameter class of interest. For any $p\in\Ht$, we have
\[
    \mathbb P_p(X_i=x) = \int_0^1 b_{t_i,x}(q)p(q)\dd q \leq p_{\max}\int_0^1 b_{t_i,x}(q)\dd q
    = \frac{p_{\max}}{t_i+1}.
\]
Therefore,
\[
    \Var_p\left\{\frac{\psi_{k,t_i}(X_i)}{\sqrt{\lambda_{t_i,k}}}\right\} \leq p_{\max}\lambda_{t_i,k}^{-1},
\]
and subsequently
\[
    \Var_p(\widehat\theta_k) \leq \frac{p_{\max}}{\cI_k}.
\]
Thus, $\cI_k$ continues to control the variance uniformly over $p\in\Ht$.

\begin{figure}
    \centering
    \includegraphics[width=\linewidth]{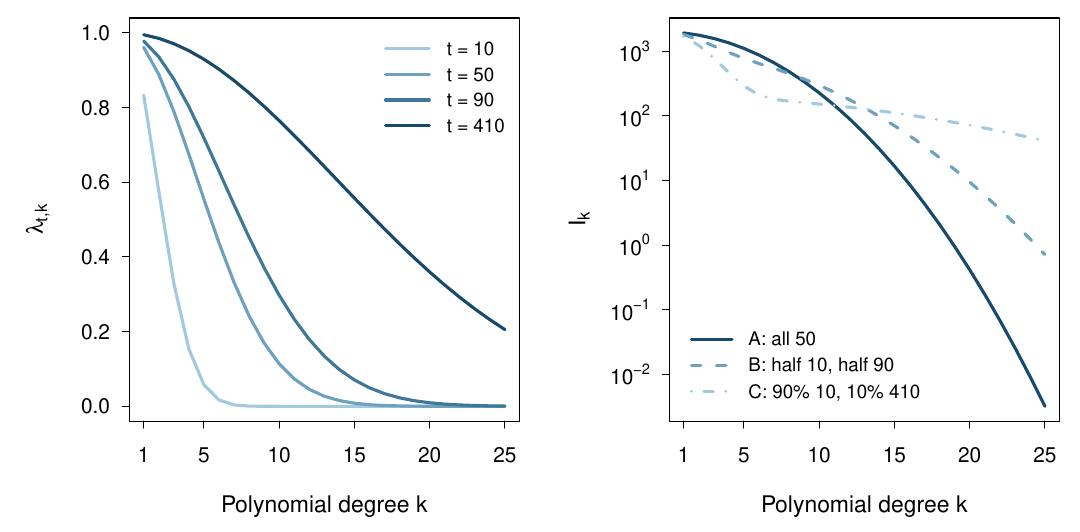}
    \caption{(a) $\lambda_{t,k}$ for $t\in\{10,50,90,410\}$, representing individual contributions to $\cI_k$. (b) $\cI_k$ for three trial allocations with $n=2{,}000$ and $\sum_{i=1}^n t_i=100{,}000$: (A) all observations have $t_i=50$; (B) half have $t_i=10$, half have $t_i=90$; (C) $90\%$ have $t_i=10$ and $10\%$ have $t_i=410$. Evenly spread trials can provide more information at lower degree; concentrating trials can preserve more information at higher degrees.}
    \label{fig:trial-information}
\end{figure}

\begin{remark}
    It is also helpful to note that
    \[
        \log\lambda_{t,k}^{-1}\asymp\frac{k^2}{t}, \qquad 1\leq k\leq\rho t.
    \]
    As a result, an observation with $t_i$ large relative to $k^2$ contributes nearly one unit to $\cI_k$, whereas one with $t_i$ small relative to $k^2$ contributes little, and contributes nothing when $t_i<k$. Thus, information at higher degrees increasingly depends on observations with larger trial counts, as illustrated in Figure~\ref{fig:trial-information}(a).
\end{remark}

\paragraph{Fundamental limits.}
With $\cI_k$, we can now state the lower bound under trial heterogeneity. We impose the following condition on the largest trial parameter, which ensures that the set of admissible polynomial degrees under trial heterogeneity
\[ 
    \cM_{\mathbf{t}} \equiv \left\{ m\in\mathbb N: m_0\leq m\leq\floor{\rho T} \right\},
\]
is nonempty.

\begin{assumption}\label{assm:constants_hetero}
We assume $T\geq\lceil m_0/\rho\rceil$,
where $m_0$ and $\rho$ are defined in Assumption \ref{assm:constants_homo}.
\end{assumption}

The following theorem characterizes the lower bound. 

\begin{theorem}[Lower bound under trial heterogeneity]\label{thm:heterogeneous_lower_bound}
Fix $s>0$, $\delta\in(0,1/2)$, and $u\in[\delta,1-\delta]$. Suppose that $L,p_{\max}>1$ and Assumption~\ref{assm:constants_hetero} holds. Then
\[
    \cR_{n,\mathbf t} \gtrsim \min_{m \in \cM_{\mathbf{t}}}\left\{m^{-s} + \sqrt{\frac{m}{\cI_m}}\right\}.
\]
\end{theorem}

The first term again describes the bias, and the second the variance in terms of the effective sample size at an admissible degree $m$. The result leverages the smooth perturbation framework in Section~\ref{sec:fundamental_limits};
under trial heterogeneity, $\cI_k$ replaces $n\lambda_{t,k}$ in the quadratic form governing the difference between the observed distributions. A caveat, however, is that $\cI_k$ now can change sharply across nearby degrees, so the homogeneous argument for controlling this quadratic form after multiplication by $w_R$ does not apply directly. We give the details for the modified perturbation construction in Appendix~\ref{sec:trial_heterogeneity_details}.

\paragraph{Attainability.} Following our earlier discussion on the Legendre coefficient estimator under heterogeneous trials, the lower bound in Theorem \ref{thm:heterogeneous_lower_bound} can be attained by replacing the homogeneous Legendre coefficient estimator \eqref{eq:coefficient_estimator} with its heterogeneous counterpart \eqref{eq:legendre_coeff_estimator_het}.

\begin{definition}[Legendre--Hahn series estimator under trial heterogeneity]
\label{def:heterogeneous_estimator}
Let $\eta$ be the regularization function from Definition
\ref{def:legendre_hahn_estimator}. For an integer
$1\leq m\leq\floor{T/2}$, we define the Legendre--Hahn series estimator under trial heterogeneity as
\[
    \widehat p_m^{\mathrm{het}}(q)
    \equiv
    \sum_{k=0}^{2m}
    \eta(k/m)\widehat\theta_k\phi_k(q),
    \qquad
    q\in[0,1],
\]
where $\widehat\theta_k$ is defined in \eqref{eq:legendre_coeff_estimator_het}.
\end{definition}

\begin{theorem}[Upper bound under trial heterogeneity]
\label{thm:heterogeneous_upper_bound}
Suppose that $p\in\Ht$ and $1\leq m\leq\floor{T/2}$. Then
\[
    \E_p\bigl|\widehat p_m^{\mathrm{het}}(u)-p(u)\bigr| \lesssim m^{-s} + \left\{\sum_{k=0}^{2m}\frac{\phi_k(u)^2}{\mathcal I_k}\right\}^{1/2}.
\]
\end{theorem}

The bias bound remains unchanged because the coefficient estimates remain unbiased; the difference lies in the variance, where $\mathcal I_k^{-1}$ now replaces $(n\lambda_{t,k})^{-1}$.

\paragraph{Overall minimax rate.} Combining the upper bound with the lower bounds yields the following characterization for the minimax rates under trial heterogeneity.

\begin{corollary}[Minimax rate under trial heterogeneity]
\label{cor:heterogeneous_minimax_rate}
Fix $u\in[\delta,1-\delta]$. Suppose that $L,p_{\max}>1$ and Assumption \ref{assm:constants_hetero} holds. Then
\[
    \cR_{n,\mathbf t}
    \asymp
    \min_{m \in \cM_{\mathbf{t}}}
    \left\{
        m^{-s}
        +
        \sqrt{\frac{m}{\cI_m}}
    \right\}.
\]
\end{corollary}
 
The estimator with degree $\floor{m/2}$ has risk bounded by a constant multiple of $m^{-s}+\sqrt{m/\cI_m}$, hence matching the upper bound with the lower bound; taking the minimum over $m \in \cM_{\mathbf{t}}$ gives the result. 

Unlike in the homogeneous setting, the minimax rate depends on $(\cI_k)_{k=0}^{T}$ rather than on a single trial parameter; it therefore does not generally reduce to the three regimes in
Corollary~\ref{cor:simplified_minimax_rate}. We thus conclude this section by giving a more intuitive characterization of the minimax rate under a simple condition that the number of observations with large trial parameters decreases polynomially.

\begin{corollary}[Polynomially decreasing trial parameters]
\label{cor:polynomial_trial_profile}
Suppose that the conditions of Corollary~\ref{cor:heterogeneous_minimax_rate} hold and, for some fixed $\beta>0$,
\[
    t_{[i]}\asymp(n/i)^{1/\beta},
    \qquad i=1,\ldots,n,
\]
where $t_{[i]}$ denotes the $i$-th largest trial parameter. Then, for all sufficiently large $n$,
\[
    \cR_{n,\mathbf t}
    \asymp n^{-s/(2s+2\beta+1)}.
\]
\end{corollary}
Here, larger $\beta$ means fewer observations with large trial parameters. Compared with the usual nonparametric rate $n^{-s/(2s+1)}$, the additional $2\beta$ can be interpreted as reflecting how the effective sample size decreases, as the degree needed to approximate the density increases.


\section{Discussion}
In this paper, we have characterized the minimax rates for pointwise estimation of a smooth binomial mixing density. We derived complementary intermediate $t$ and small $t$ lower bounds, and showed that the resulting lower bound is attained by a regularized orthogonal series estimator. We further extended the analysis to heterogeneous trial parameters, where the relevant notion of information becomes degree-specific and depends on the full set of trials.

There are several natural directions for future work. One is to incorporate covariates; as mentioned in \cite{lee2025learning}, if $C_i$ denotes the vector of covariates for unit $i$, one natural model to consider is
\[
   X_i \mid Q_i = q \sim \operatorname{Bin}(t_i, q), \qquad
   \operatorname{logit}(Q_i) \mid C_i=c
   \sim N(m(c), \sigma^2),
\]
following \cite{ignatiadis2019covariate}, and to study how assumptions on $m(\cdot)$ affect the attainable rates. Relatedly, \cite{chen2025compound} study compound estimation of heterogeneous binomial means and show how covariate-based machine-learning predictions can be incorporated into shrinkage estimators. It would be interesting to understand how covariate information can similarly be combined with the smooth population structure in our setting. 


A second intriguing question is what our results mean in the context of existing optimality results for binomial mixtures, and whether the analyses can be transported in either direction. For instance, one could consider the settings of \cite{tian2017learning,vinayak2019maximum} while additionally imposing smoothness on the mixing distribution, and study how smoothness changes the corresponding rates. Conversely, one could also study the performance of estimators developed for smooth density estimation under the distributional losses considered in that literature.

Finally, in some applications, such as ecological abundance models
\citep{royle2004n,wenger2008estimating,dail2011models,wu2015bayesian},
the binomial size may itself be a latent random variable. 
Understanding how uncertainty in the binomial size changes the minimax rates would be an interesting avenue for future work.

\paragraph{Acknowledgments.}
We thank Lucas Kania for helpful comments during the development of this work. The lead author used ChatGPT to identify relevant literature, assist with mathematical proofs, and improve the exposition of this manuscript. For example, GPT 5.5 was used to locate literature concerning Bernstein–Durrmeyer operators and to assist with calculations in some of the technical lemmas. The manuscript was prepared through an iterative process in which the author first drafted the text, used GPT to check for potential errors and suggest improvements, and then reviewed the suggestions and made all final decisions. We welcome any further questions regarding the role of AI in this work.

\bibliography{references}
\appendix

\section{Fundamental limits under trial heterogeneity}
\label{sec:trial_heterogeneity_details}

For consistency with the analysis under trial homogeneity, we keep the terms ``intermediate $t$ lower bound'' and ``small $t$ lower bound''. The former is obtained by controlling the divergence between the observed product distributions, and the latter makes those
distributions exactly identical. Importantly, these terms distinguish the two lower-bound constructions, not separate regimes of the heterogeneous trials.

\paragraph{Extended perturbation framework.}
We begin by extending the framework in Section~\ref{sec:shared_framework} to heterogeneous trials. Define the \emph{aggregate} Bernstein--Durrmeyer operator as
\[
    \cBD_{\mathbf t}
    \equiv
    \sum_{i=1}^n \cBD_{t_i},
\]
where $\cBD_{t_i}$ is Bernstein--Durrmeyer operator for observation $i$. 
Much of the homogeneous analysis continues to hold; 
first, $\cBD_{\mathbf t}$ is diagonalized by the same basis.

\begin{proposition}
\label{prop:aggregate_eigendecomposition}
For every $k\geq0$, $\cBD_{\mathbf t}\phi_k=\cI_k\phi_k$.
\end{proposition}
Hence, $\cI_k$ that measures the effective sample size for estimating the $k$-th Legendre coefficient also appears as the eigenvalue of the aggregate operator. In the homogeneous case, where $t_i=t$ for every $i$, we recover
\[
    \cBD_{\mathbf t}=n\cBD_t,
    \qquad
    \cI_k=n\lambda_{t,k},
\]
so Proposition~\ref{prop:aggregate_eigendecomposition} reduces directly to Lemma~\ref{lem:bernstein_durrmeyer_eigen}.

We now leverage Proposition~\ref{prop:aggregate_eigendecomposition} to extend the intermediate $t$ and small $t$ lower-bound constructions. If we consider analogous perturbed densities as in Section \ref{sec:shared_framework}, the corresponding probability mass functions for observation $i$ are
\[
    q_{i,\pm,\Delta}(x) = \frac{1}{t_i+1} \pm \Delta(\cB_{t_i}g)(x), \qquad x=0,\ldots,t_i.
\]
Therefore, we have
\[
    \chi^2(q_{i,+,\Delta},q_{i,-,\Delta}) \lesssim \Delta^2(t_i+1) \sum_{x=0}^{t_i}(\cB_{t_i}g)(x)^2 = \Delta^2\ip{g}{\cBD_{t_i}g}_{L_2},
\]
whenever $p_{-,\Delta}\geq c_0$. Since the observations are still independent, tensorization yields
\[
    \chi^2\left(\bigotimes_{i=1}^n q_{i,+,\Delta}, \bigotimes_{i=1}^n q_{i,-,\Delta}\right) \leq \exp\left\{C\Delta^2\sum_{i=1}^n \ip{g}{\cBD_{t_i}g}_{L_2}\right\}-1 = \exp\left\{C\Delta^2\ip{g}{\cBD_{\mathbf t}g}_{L_2}\right\}-1.
\]
Hence, under trial heterogeneity, the corresponding quadratic form governing the $\chi^2$ divergence is
$\ip{g}{\cBD_{\mathbf t}g}_{L_2}$. Proposition \ref{prop:aggregate_eigendecomposition} then implies that, for $g$ in \eqref{eq:g},
\[
    \ip{g}{\cBD_{\mathbf t}g}_{L_2} = \sum_{k=0}^m \cI_k a_k^2.
\]
Thus the optimal perturbation satisfying $g(u)=1$ is given by
\[
    g^\star_{\mathrm{het}}(q) = \frac{\sum_{k=0}^m\cI_k^{-1}\phi_k(u)\phi_k(q)}{\sum_{k=0}^m\cI_k^{-1}\phi_k(u)^2},
\]
whenever $m\leq T$.

This again suggests natural strategies for both lower bounds. For the small $t$ lower bound, since $\cI_k=0$ for every $k>T$, the terms above degree $T$ are completely unobservable from the heterogeneous sample. Thus placing the perturbation sufficiently far above degree $T$ yields the desired bound.

For the intermediate $t$ lower bound, the form of $g^\star_{\mathrm{het}}$ suggests a direct analogue of the homogeneous construction where one replaces $\lambda_{t,k}^{-1}$ with $\cI_k^{-1}$ in
\eqref{eq:pert_stat}. As we show next, however, this construction need not remain well-controlled after multiplication by $w_R$. We therefore develop an alternative construction that sacrifices exact optimality but still remains valid under arbitrary trial heterogeneity.

\paragraph{Intermediate $t$ lower bound.} 
Recall that the intermediate $t$ lower bound under trial homogeneity relied on (local) comparability of the eigenvalues across the degree shifts introduced by $w_R$ (Lemma~\ref{lem:operator_bound}). The analogous argument under trial heterogeneity would require 
\[ 
    \cI_\ell\lesssim\cI_k, \qquad |\ell-k|\leq2R.
\] 
This, however, need not hold under arbitrary trial heterogeneity. For example, if many observations have $t_i=k-1$, then they contribute to $\cI_{k-1}$ but nothing to $\cI_k$, so $\cI_{k-1}/\cI_k$ can be arbitrarily large. Multiplication by $w_R$ can therefore move a component to a nearby lower degree with a much larger effective sample size, causing $\ip{g}{\cBD_{\mathbf t}g}_{L_2}$ to blow up. 

We avoid this difficulty by instead turning to a bound that requires only the monotonicity of $\cI_k$. Suppose that
\[
    f(q) \equiv \sum_{k\in\cJ}a_k\phi_k(q),
\]
is supported on a block $\cJ$ chosen so that every degree appearing in $w_Rf$ is strictly larger than $m$. Writing
\[
    w_Rf = \sum_{\ell>m}c_\ell\phi_\ell,
\] 
Proposition \ref{prop:aggregate_eigendecomposition} and the monotonicity of $\cI_k$ give
\[
    \ip{w_Rf}{\cBD_{\mathbf t}(w_Rf)}_{L_2}
    = \sum_{\ell>m}\cI_\ell c_\ell^2
    \leq\cI_m\sum_{\ell>m}c_\ell^2
    = \cI_m\norm{w_Rf}_{L_2}^2
    \lesssim \cI_m\norm{f}_{L_2}^2.
\]
Thus, after shifting the block sufficiently far above degree $m$, the quadratic form is controlled by $\cI_m$ and the $L_2$ norm of $f$. We therefore seek a polynomial with small $L_2$ norm subject to $f(u)=1$.

For a fixed block $\cJ$, Cauchy--Schwarz gives
\[
    1 = \bigg\{\sum_{k\in\cJ}a_k\phi_k(u)\bigg\}^2
    \leq \bigg\{\sum_{k\in\cJ}a_k^2\bigg\}\bigg\{\sum_{k\in\cJ}\phi_k(u)^2\bigg\},
\]
with equality when $a_k$ is proportional to $\phi_k(u)$. Hence, among polynomials supported on $\cJ$ and equal to one at $u$, 
\[
    \frac{\sum_{k\in\cJ}\phi_k(u)\phi_k(q)}{\sum_{k\in\cJ}\phi_k(u)^2},
\]
has the smallest $L_2$ norm.\footnote{This unweighted construction also applies to homogeneous trials. We use the inverse-eigenvalue weighting there because it is optimal before multiplication by $w_R$, after which the quadratic form remains bounded by a constant multiple of $S_m(u)^{-1}$.}
We therefore choose
\[
    g_{\cJ_{\intm}^{\mathrm{het}}}(q)
    \equiv w_R(q) \frac{\sum_{k\in\cJ_{\intm}^{\mathrm{het}}}\phi_k(u)\phi_k(q)}{\sum_{k\in\cJ_{\intm}^{\mathrm{het}}}\phi_k(u)^2},
\]
where 
\[
    \cJ_{\intm}^{\mathrm{het}} \equiv \cJ(m+2R,m) = \{m+2R+1,\ldots,2m+2R\}, \qquad m \in \cM_{\mathbf{t}}.
\]
The block has width $m$, so Lemma \ref{lem:legendre_block} gives
\[
    \sum_{k\in\cJ_{\intm}^{\mathrm{het}}} \phi_k(u)^2 \asymp m.
\]
Moreover, all degrees are of order $m$, which yields the smoothness scale $m^s$. The lower endpoint $m+2R+1$ is chosen so that, after multiplication by $w_R$, every resulting degree remains strictly larger than $m$. Hence, by monotonicity, every such degree has effective sample size at most $\cI_m$.

Now the following two lemmas help characterize the largest useful amplitude.

\begin{lemma}
\label{lem:heterogeneous_perturbation}
Suppose Assumption~\ref{assm:constants_hetero} holds. Then, for every $m \in \cM_{\mathbf{t}}$, uniformly over $u\in[\delta,1-\delta]$,
\[
    \|g_{\cJ_{\intm}^{\mathrm{het}}}\|_{C^s([0,1])}
    \lesssim
    m^s.
\]
\end{lemma}

\begin{lemma}
\label{lem:heterogeneous_operator_bound}
Suppose Assumption~\ref{assm:constants_hetero} holds. Then, for every $m \in \cM_{\mathbf{t}}$, uniformly over $u\in[\delta,1-\delta]$,
\[
    \langle g_{\cJ_{\intm}^{\mathrm{het}}}, \cBD_{\mathbf t}
        g_{\cJ_{\intm}^{\mathrm{het}}}\rangle_{L_2}
    \lesssim
    \frac{\cI_m}{m}.
\]
\end{lemma}

To see why Lemma \ref{lem:heterogeneous_operator_bound} holds, note that orthonormality implies
\[
    \left\|
        \frac{
            \sum_{k\in\cJ_{\intm}^{\mathrm{het}}}
            \phi_k(u)\phi_k
        }{
            \sum_{k\in\cJ_{\intm}^{\mathrm{het}}}
            \phi_k(u)^2
        }
    \right\|_{L_2}^2
    =
    \bigg\{
        \sum_{k\in\cJ_{\intm}^{\mathrm{het}}}
        \phi_k(u)^2
    \bigg\}^{-1}
    \asymp
    \frac{1}{m}.
\]
Since every degree after multiplication by $w_R$ is larger than $m$, we have from monotonicity of $\cI_m$ that
\[
    \langle g_{\cJ_{\intm}^{\mathrm{het}}}, \cBD_{\mathbf t}
        g_{\cJ_{\intm}^{\mathrm{het}}}\rangle_{L_2}
    \lesssim
    \cI_m
    \left\|
        \frac{
            \sum_{k\in\cJ_{\intm}^{\mathrm{het}}}
            \phi_k(u)\phi_k
        }{
            \sum_{k\in\cJ_{\intm}^{\mathrm{het}}}
            \phi_k(u)^2
        }
    \right\|_{L_2}^2
    \lesssim
    \frac{\cI_m}{m}.
\]

We can thus define the perturbed densities as follows.

\begin{definition}
For a sufficiently small constant $c>0$, the perturbed densities for the intermediate $t$ lower bound under trial heterogeneity are
\[
    p_{\pm,\Delta_m^{\mathrm{het}}}(q)
    \equiv
    1\pm
    \Delta_m^{\mathrm{het}}g_{\cJ_{\intm}^{\mathrm{het}}}(q), \qquad \Delta_m^{\mathrm{het}}
    \equiv
    c\min\left\{m^{-s}, \sqrt{\frac{m}{\cI_m}}\right\}.
\] 
\end{definition}

By Lemma \ref{lem:heterogeneous_perturbation}, the first term in $\Delta_m^{\mathrm{het}}$ ensures that $p_{\pm,\Delta_m^{\mathrm{het}}}\in\Ht$ for sufficiently small $c > 0$. 
The second term helps as
\[
    \sum_{i=1}^n
    \chi^2(q_{i,m,+},q_{i,m,-})
    \lesssim
    (\Delta_m^{\mathrm{het}})^2
    \frac{\cI_m}{m}
    \lesssim c^2,
\]
and the product divergence is again bounded by $\exp(Cc^2)-1$. The following corollary gives the intermediate $t$ lower bound under trial heterogeneity. 

\begin{corollary}[Intermediate $t$ lower bound under trial heterogeneity]
\label{cor:heterogeneous_statistical_lower_bound}
Fix $u\in[\delta,1-\delta]$. Suppose that $L,p_{\max}>1$ and Assumption \ref{assm:constants_hetero} holds. Then 
\[
    \cR_{n,\mathbf t}
    \gtrsim
    \max_{m \in \cM_{\mathbf{t}}}
    \min\left\{
        m^{-s},
        \sqrt{\frac{m}{\cI_m}}
    \right\}.
\]
\end{corollary}

\paragraph{Small $t$ lower bound.} 
Since every Bernstein basis function now has degree at most $T$, we can place the perturbation above degree $T+2R$ so that every degree remaining after multiplication by $w_R$ is strictly larger than $T$. We thus choose
\[
    g_{\cJ_{\sml}^{\mathrm{het}}}(q)
    \equiv
    w_R(q)
    \frac{
        \sum_{k\in\cJ_{\sml}^{\mathrm{het}}}
        \phi_k(u)\phi_k(q)
    }{
        \sum_{k\in\cJ_{\sml}^{\mathrm{het}}}
        \phi_k(u)^2
    },
\]
where
\[
    \cJ_{\sml}^{\mathrm{het}} \equiv \cJ(T+2R,W_0) = \{T+2R+1,\ldots,T+2R+W_0\}.
\]




\begin{definition}[Densities for the small $t$ lower bound under trial heterogeneity]
For a sufficiently small constant $c>0$, define
\[
    p_{\pm,\Delta_T}(q)
    \equiv
    1\pm \Delta_T g_{\cJ_{\sml}^{\mathrm{het}}}(q),
    \qquad
    \Delta_T\equiv cT^{-s}.
\]
\end{definition}

The same argument as in Section \ref{sec:trial_only_lower_bound} shows that $p_{\pm,\Delta_T}\in\Ht$ and that they differ at $u$ by $2cT^{-s}$. Since $\cB_{t_i}g_{\cJ_{\sml}^{\mathrm{het}}}=0$ for every $i$, the two densities induce exactly the same joint distribution of the observed sample and yield the following lower bound.

\begin{corollary}[Small $t$ lower bound under trial heterogeneity]
\label{cor:heterogeneous_trial_lower_bound}
Fix $u\in[\delta,1-\delta]$ and suppose that $L,p_{\max}>1$. Then, we have
\[
    \cR_{n,\mathbf t} \gtrsim T^{-s}.
\]
\end{corollary}

Combining the two lower bounds and applying the same monotonicity argument as in Corollary \ref{cor:lower_bound_bias_variance} yields
\[
    \cR_{n,\mathbf t}
    \gtrsim
    \min_{m \in \cM_{\mathbf{t}}}
    \left\{
        m^{-s}
        +
        \sqrt{\frac{m}{\cI_m}}
    \right\}.
\]




\section{Proofs}
\label{sec:proofs}

\subsection{Proofs for the smooth perturbation framework}
\label{sec:proofs_shared_construction}

\begin{proof}[Proof of Lemma \ref{lem:bernstein_durrmeyer_eigen}]
We refer the readers to \cite{Derriennic1981}.
\end{proof}

\begin{proof}[Proof of Lemma \ref{lem:legendre_block}]
The proof proceeds in two parts, first some preliminary calculations and then bounding the actual quantity of interest.

\paragraph{Preliminary calculations.} Write $2u-1 = \cos\theta$. Since $u\in[\delta,1-\delta]$, we have $\theta\in[\theta_\delta,\pi-\theta_\delta]$ for some $\theta_\delta>0$ depending only on $\delta$. As $k \to \infty$, Theorem 8.21.8 of \cite{szego75} yields
\[
    P_k(\cos\theta) = \left(\frac{2}{\pi k\sin\theta}\right)^{1/2} \cos\left(\left(k+\frac12\right)\theta-\frac{\pi}{4}\right) + O_\delta(k^{-3/2}),
\]
uniformly for $\theta\in[\theta_\delta,\pi-\theta_\delta]$. As a result, for the shifted Legendre polynomial $\phi_k$, we have
\[
    \phi_k(u) = \frac{2}{\sqrt{\pi\sin\theta}}
    \cos\left(\left(k+\frac12\right)\theta-\frac{\pi}{4}\right) + O_\delta(k^{-1}) \equiv \frac{2}{\sqrt{\pi\sin\theta}}\cos\left(a_k\right) + O_\delta(k^{-1}),
\]
uniformly for $\theta\in[\theta_\delta,\pi-\theta_\delta]$ as $k\to\infty$. 
If we define $R_k(\theta)$ to be the exact difference
\[
    R_k(\theta) \equiv \phi_k(u) - \frac{2}{\sqrt{\pi\sin\theta}}\cos(a_k),
\]
then the asymptotic expansion implies that there exist $k_\delta\geq1$ and $A_\delta<\infty$ such that
\[
    \sup_{\theta\in[\theta_\delta,\pi-\theta_\delta]}
    |R_k(\theta)|
    \leq \frac{A_\delta}{k}, \qquad k\geq k_\delta.
\]
For each of the finitely many integers $0\leq k<k_\delta$, the function $R_k$ is continuous on the compact interval $[\theta_\delta,\pi-\theta_\delta]$. Therefore,
\[
    C_{\delta,0} \equiv \max\left\{2A_\delta,\,\max_{0\leq k<k_\delta}(k+1)\sup_{\theta\in[\theta_\delta,\pi-\theta_\delta]}|R_k(\theta)|\right\}
    <\infty.
\]
Since $k^{-1}\leq2(k+1)^{-1}$ for $k\geq1$, it follows for every $k\in\mathbb N_0$ that
\[
    \phi_k(u) = \frac{2}{\sqrt{\pi\sin\theta}}\cos(a_k)
    +R_k(\theta), \qquad |R_k(\theta)|
    \leq \frac{C_{\delta,0}}{k+1},
\]
uniformly over $\theta \in [\theta_\delta,\pi-\theta_\delta]$. 
Consequently, for every $a\in\mathbb N_0$ and $w\geq1$,
\[
    \sum_{k\in\cJ(a,w)}R_k(\theta)^2
    \leq C_\delta\sum_{k=0}^{\infty}(k+1)^{-2}
    \leq C_\delta.
\]

Next, fix $a\in\mathbb N_0$ and $w\geq1$. By the identity $2\cos^2(a_k)=1+\cos(2a_k)$, we have
\begin{align*}
    \sum_{k\in\cJ(a,w)}\cos^2(a_k)
    &=
    \frac{w}{2}
    +
    \frac12
    \sum_{k=a+1}^{a+w}\cos(2a_k)\\
    &=
    \frac{w}{2}
    +
    \frac12
    \sum_{k=a+1}^{a+w}
    \cos\left(
        2k\theta+\theta-\frac{\pi}{2}
    \right)\\
    &=
    \frac{w}{2}
    +
    \frac12\operatorname{Re}
    \left\{
        e^{i(\theta-\pi/2)}
        \sum_{k=a+1}^{a+w}e^{2ik\theta}
    \right\},
\end{align*}
where
\begin{align*}
    \left|
        \sum_{k=a+1}^{a+w}e^{2ik\theta}
    \right|
    =
    \left|
        e^{2i(a+1)\theta}
        \sum_{j=0}^{w-1}e^{2ij\theta}
    \right| =
    \left|
        e^{2i(a+1)\theta}
        \frac{1-e^{2iw\theta}}{1-e^{2i\theta}}
    \right| \leq
    \frac{2}{|1-e^{2i\theta}|}
    =
    \frac{1}{|\sin\theta|}.
\end{align*}
It follows that
\[
    \left|
        \sum_{k\in\cJ(a,w)}
        \cos^2(a_k)
        -
        \frac{w}{2}
    \right|
    \leq
    \frac{1}{2|\sin\theta|}.
\]

\paragraph{Bounding the quantity of interest.} We now bound the quantity of interest
$\sum_{k\in\cJ(a,w)}\phi_k(u)^2$. For the upper bound,
$(x+y)^2\leq2x^2+2y^2$ yields
\begin{align*}
    \sum_{k\in\cJ(a,w)}\phi_k(u)^2
    &\leq
    \frac{8}{\pi\sin\theta}
    \sum_{k\in\cJ(a,w)}\cos^2(a_k)
    +
    2\sum_{k\in\cJ(a,w)}R_k(\theta)^2\\
    &\leq
    \frac{8}{\pi\sin\theta}
    \left(
        \frac{w}{2}
        +
        \frac{1}{2\sin\theta}
    \right)
    +
    C_\delta\\
    &=
    \frac{4w}{\pi\sin\theta}
    +
    \frac{4}{\pi\sin^2\theta}
    +
    C_\delta \leq C_\delta w,
\end{align*}
where the last inequality follows from
$\sin\theta\geq2\sqrt{\delta(1-\delta)}$ and $w\geq1$.

For the lower bound, applying $(x+y)^2\geq x^2/2-y^2$ gives
\begin{align*}
    \sum_{k\in\cJ(a,w)}\phi_k(u)^2
    &\geq
    \frac{2}{\pi\sin\theta}
    \sum_{k\in\cJ(a,w)}\cos^2(a_k)
    -
    \sum_{k\in\cJ(a,w)}R_k(\theta)^2\\
    &\geq
    \frac{2}{\pi\sin\theta}
    \left(
        \frac{w}{2}
        -
        \frac{1}{2\sin\theta}
    \right)
    -
    C_\delta\\
    &=
    \frac{w}{\pi\sin\theta}
    -
    \frac{1}{\pi\sin^2\theta}
    -
    C_\delta\\
    &\geq
    \frac{w}{\pi}
    -
    C_\delta,
\end{align*}
where the last inequality follows from $\sin\theta\leq1$ for the first term and from
$\sin\theta\geq2\sqrt{\delta(1-\delta)}$ to absorb
$1/(\pi\sin^2\theta)$ into $C_\delta$.

If we choose
$
    W_0(\delta)
    \geq
    \max\left\{
        1,
        \left\lceil2\pi C_\delta\right\rceil
    \right\},
$
then, for every $w\geq W_0(\delta)$,
\[
    \frac{w}{\pi}-C_\delta
    \geq
    \frac{w}{2\pi}.
\]
Combining this with the upper bound gives
\[
    c_\delta|\cJ(a,w)|
    \leq
    \sum_{k\in\cJ(a,w)}\phi_k(u)^2
    \leq
    C_\delta|\cJ(a,w)|,
\]
where one may take $c_\delta=1/(2\pi)$. Since
$|\cJ(a,w)|=w$, this proves the result.
\end{proof}

We note a consequence of Lemma \ref{lem:legendre_block}, which is the uniform bound
\begin{equation}
\label{eq:proof_uniform_legendre_bound}
    \sup_{u\in[\delta,1-\delta]}
    \sup_{k\geq0}
    |\phi_k(u)|
    \leq
    C_\delta.
\end{equation}
This follows since $\phi_0\equiv1$, and for $k \geq 1$, we can apply the upper bound
in Lemma \ref{lem:legendre_block} to
$\{k,\ldots,k+W_0-1\}$.

\subsection{Proofs for the intermediate $t$ lower bound}
\label{sec:proofs_statistical_lower_bound}

We first verify the two basic properties of the constructed perturbation for completeness. For notational simplicity, write
\begin{equation}\label{eq:kbstat}
    K_{\cJ_{\intm}}(u,q) \equiv \sum_{k\in\cJ_{\intm}}\lambda_{t,k}^{-1}\phi_k(u)\phi_k(q).
\end{equation}

Since $w_R(u)=1$, $ g_{\cJ_{\intm}}(u)=1$. Further, we have 
\begin{align*}
    \int_0^1g_{\cJ_{\intm}}(q)\dd q
    &=
    \frac{1}{K_{\cJ_{\intm}}(u,u)}
    \sum_{k=2R+1}^m
    \lambda_{t,k}^{-1}\phi_k(u)
    \ip{w_R\phi_k}{\phi_0}_{L_2} \\
    &=
    \frac{1}{K_{\cJ_{\intm}}(u,u)}
    \sum_{k=2R+1}^m
    \lambda_{t,k}^{-1}\phi_k(u)
    \ip{\phi_k}{w_R}_{L_2} =0,
\end{align*}
where the second equality follows from $\phi_0 \equiv 1$ and self-adjointness of the multiplication by a real-valued function, and the last from $w_R$ having degree $2R<k$ for every $k\in\cJ_{\intm}$.

Throughout the rest of the analysis, the following result on the eigenvalues will prove handy.

\begin{lemma}\label{lem:inverse_eig_ratio}
    For $0\le\ell\le m\le\rho t$, we have
    \[
        \exp\left(-\frac{2}{1-\rho}\frac{\ell m}{t}\right) 
        \leq \frac{\lambda_{t,m-\ell}^{-1}}{\lambda_{t,m}^{-1}}
        \leq \exp\left(-\frac{1-\rho}{1+\rho}\frac{\ell m}{t}\right) \leq 1.
    \]
\end{lemma}

\begin{proof}
    The result is trivial when $\ell=0$, so suppose that $1\le \ell\le m$. Note that
    \[
        \log\frac{\lambda_{t,m}^{-1}}{\lambda_{t,m-\ell}^{-1}} = \sum_{j=m-\ell}^{m-1}\log\left(\frac{t+j+2}{t-j}\right) 
        = \sum_{j=m-\ell}^{m-1}\log\left(1+\frac{2j+2}{t-j}\right).
    \]
    Since $(2j+2)/(t-j) \geq 0$ for every $j \in \{m-\ell,\ldots,m-1\}$, by 
    \[
        \frac{x}{1+x}\leq \log(1+x) \leq x,
        \qquad x \geq 0,
    \]
    with $x = (2j+2)/(t-j)$, we have 
    \[
        \frac{2j+2}{t+j+2}
        \leq \log\left(1+\frac{2j+2}{t-j}\right)
        \leq \frac{2j+2}{t-j}.
    \]
    We further claim that
    \[
        \frac{1-\rho}{1+\rho}\frac{2j+2}{t}
        \leq \log\left(1+\frac{2j+2}{t-j}\right)
        \leq \frac{1}{1-\rho}\frac{2j+2}{t}.
    \]
    For the upper bound, note that we have $(1-\rho)t \leq t - m \leq t-j$. Hence
    \[
        \frac{2j+2}{t-j} \leq \frac{2j+2}{(1-\rho)t}.
    \]
    For the lower bound, again
    since $j+1\le m\le\rho t$,
    \[
        x = \frac{2(j+1)}{t-j}
        \leq \frac{2\rho t}{(1-\rho)t}
        = \frac{2\rho}{1-\rho}.
    \]
    Consequently,
    \[
        \frac{x}{1+x}
        \geq \frac{x}{1+2\rho/(1-\rho)}
        = \frac{1-\rho}{1+\rho}x 
        \geq \frac{1-\rho}{1+\rho}\frac{2j+2}{t},
    \]
    so the claimed inequalities hold. Finally, note that
    \[
    \sum_{j=m-\ell}^{m-1}(j+1) = \frac{\ell(2m-\ell+1)}{2},
    \]
    and, since $1 \leq \ell \leq m$,
    \[
        \frac{\ell m}{2} 
        \leq \sum_{j=m-\ell}^{m-1}(j+1) 
        \leq \ell m.
    \]
    It follows that
    \[
        \frac{1-\rho}{1+\rho}\frac{\ell m}{t}
        \leq \log \frac{\lambda_{t,m}^{-1}}{\lambda_{t,m-\ell}^{-1}}
        \leq \frac{2}{1-\rho}\frac{\ell m}{t}.
    \]
    Exponentiating the negative of these inequalities yields the desired.
\end{proof}

\begin{proof}[Proof of Lemma \ref{lem:localized_block_comparison}]
By definition, we can write
\[
S_m(u)
=
\sum_{k\in \cJ_{\intm}}\lambda_{t,k}^{-1}\phi_k(u)^2
+
\sum_{k=0}^{2R}\lambda_{t,k}^{-1}\phi_k(u)^2.
\]
For each fixed $k\leq2R$, the function $\phi_k$ is bounded on $[0,1]$. Further, $\lambda_{t,k}^{-1}$ is nonincreasing in $t$ and hence is bounded uniformly over $t\geq t_0$.
It follows that
\[
    \sum_{k=0}^{2R}\lambda_{t,k}^{-1}\phi_k(u)^2\leq C.
\]
We also have
\[
    \sum_{k\in \cJ_{\intm}}\lambda_{t,k}^{-1}\phi_k(u)^2
    \geq \sum_{k=2R+1}^{m}\phi_k(u)^2
    \geq c_\delta(m-2R) 
    \geq c_\delta W_0,
\]
where the first inequality follows from $\lambda_{t,k}^{-1}\geq1$,
the second from Lemma~\ref{lem:legendre_block}, and the third from
$|\cJ_{\intm}|=m-2R\geq W_0$. Hence
\[
    \sum_{k\in \cJ_{\intm}}\lambda_{t,k}^{-1}\phi_k(u)^2
    \leq S_m(u)
    \leq \sum_{k\in \cJ_{\intm}}\lambda_{t,k}^{-1}\phi_k(u)^2+C
    \lesssim \sum_{k\in \cJ_{\intm}}\lambda_{t,k}^{-1}\phi_k(u)^2.
\]

Now let
\[
    \tilde g(q) \equiv \sum_{k=2R+1}^{m}c_k\phi_k(q), \qquad \tilde g(u)=1.
\]
Then we have
\[
    1 = \left(\sum_{k=2R+1}^{m}c_k\phi_k(u)\right)^2  \leq\left(\sum_{k=2R+1}^{m}\lambda_{t,k}c_k^2\right)\left(\sum_{k=2R+1}^{m}\lambda_{t,k}^{-1}\phi_k(u)^2\right)
    = \langle \tilde g,\cBD_t\tilde g\rangle_{L_2} \sum_{k\in \cJ_{\intm}}\lambda_{t,k}^{-1}\phi_k(u)^2,
\]
where the second step follows from Cauchy-Schwarz and the third from orthonormality of $\phi$ and Lemma \ref{lem:bernstein_durrmeyer_eigen}. It follows that 
\[
    \langle \tilde g,\cBD_t\tilde g\rangle_{L_2} \geq \left\{\sum_{k\in \cJ_{\intm}}\lambda_{t,k}^{-1}\phi_k(u)^2\right\}^{-1},
\]
and equality is attained for
\[
    c_k = \lambda_{t,k}^{-1}\phi_k(u)\left\{\sum_{k\in \cJ_{\intm}}\lambda_{t,k}^{-1}\phi_k(u)^2\right\}^{-1}. \qedhere
\]
\end{proof}

\begin{proof}[Proof of Lemma \ref{lem:operator_bound}]
We have by definition that
\[
    \langle g_{\cJ_{\intm}},\cBD_tg_{\cJ_{\intm}}\rangle_{L_2} 
    = \frac{\langle w_RK_{\cJ_{\intm}}, \cBD_t(w_RK_{\cJ_{\intm}})\rangle_{L_2}}{K_{\cJ_{\intm}}(u,u)^2}, 
\]
with $K_{\cJ_{\intm}}$ defined in \eqref{eq:kbstat}. We will show that
\[
    \langle w_RK_{\cJ_{\intm}},\cBD_t(w_RK_{\cJ_{\intm}})\rangle_{L_2} \lesssim K_{\cJ_{\intm}}(u,u),
\]
so that dividing by $K_{\cJ_{\intm}}(u,u)^2$ and applying Lemma \ref{lem:localized_block_comparison} gives the desired.

Let $\beta_{\ell k} \equiv \langle w_R\phi_k,\phi_\ell\rangle_{L_2}$ for $k,\ell\geq0$. Note that $\beta_{\ell k}=0$ whenever $|\ell-k|>2R$ since $w_R$ has degree $2R$. If $\ell>k+2R$, then $w_R\phi_k$ has degree strictly less than $\ell$; if $\ell<k-2R$, then $\beta_{\ell k} = \langle\phi_k,w_R\phi_\ell\rangle_{L_2}=0$ because $w_R\phi_\ell$ has degree strictly less than $k$. Hence we can write
\[
w_RK_{\cJ_{\intm}}
=
\sum_{\ell=0}^{t}c_\ell\phi_\ell,
\qquad
c_\ell=
\sum_{\substack{k \in \cJ_{\intm} \\|\ell-k|\leq2R}}
\lambda_{t,k}^{-1}\phi_k(u)\beta_{\ell k}.
\]

We next compare the eigenvalues for $|\ell-k|\leq2R$ where $0\leq\ell\leq t$ and $k\in\cJ_{\intm}$. If $\ell\geq k$, monotonicity gives
$\lambda_{t,\ell}\leq\lambda_{t,k}$.
If $\ell<k$, applying the lower bound in Lemma~\ref{lem:inverse_eig_ratio} and taking reciprocals gives
\[
    \frac{\lambda_{t,\ell}}{\lambda_{t,k}}
    \leq
    \exp\left(\frac{2(k-\ell)k}{(1-\rho)t}\right)
    \leq
    \exp\left(\frac{4R\rho}{1-\rho}\right),
\]
where the last inequality uses $k-\ell\leq2R$ and $k\leq m\leq\rho t$. It follows that $\lambda_{t,\ell}\lesssim\lambda_{t,k}$ in both cases. We thus have
\begin{align*}
    \langle w_RK_{\cJ_{\intm}},
    \cBD_t(w_RK_{\cJ_{\intm}})\rangle_{L_2}
    &= \sum_{\ell=0}^{t}\lambda_{t,\ell}c_\ell^2\\
    &\leq (4R+1)\sum_{k \in \cJ_{\intm}}
    \lambda_{t,k}^{-2}\phi_k(u)^2
    \sum_{\substack{0\leq\ell\leq t\\|\ell-k|\leq2R}}
    \lambda_{t,\ell}\beta_{\ell k}^2\\
    &\lesssim \sum_{k \in \cJ_{\intm}}\lambda_{t,k}^{-1}\phi_k(u)^2\sum_{\substack{0\leq\ell\leq t\\|\ell-k|\leq2R}}\beta_{\ell k}^2,
\end{align*}
where the second line follows from
\[
c_\ell^2
\leq
(4R+1)
\sum_{\substack{k \in \cJ_{\intm}\\|\ell-k|\leq2R}}
\lambda_{t,k}^{-2}\phi_k(u)^2\beta_{\ell k}^2,
\]
since each sum defining $c_\ell$ contains at most $4R+1$ terms and Cauchy--Schwarz, and the third from $\lambda_{t,\ell}\lesssim\lambda_{t,k}$.

It remains to analyze $\sum_{\substack{0\leq\ell\leq t\\|\ell-k|\leq2R}}\beta_{\ell k}^2$. Note that Cauchy--Schwarz and $\|\phi_k\|_{L_2}=1$ implies
\[
    |\beta_{\ell k}| \leq \|w_R\phi_k\|_{L_2}\|\phi_\ell\|_{L_2} \leq\|w_R\|_\infty.
\]
It follows that
\[
\sum_{\substack{0\leq\ell\leq t\\|\ell-k|\leq2R}}\beta_{\ell k}^2 \leq(4R+1)\|w_R\|_\infty^2 \lesssim1,
\]
where the last inequality follows from
$\|w_R\|_\infty\leq\{4\delta(1-\delta)\}^{-R}$.

We therefore have
\[
\langle w_RK_{\cJ_{\intm}},
\cBD_t(w_RK_{\cJ_{\intm}})\rangle_{L_2}
\lesssim
\sum_{k \in \cJ_{\intm}}\lambda_{t,k}^{-1}\phi_k(u)^2
=
K_{\cJ_{\intm}}(u,u). \qedhere
\]
\end{proof}

\begin{proof}[Proof of Lemma \ref{lem:holder_bound}]
Note that, for $0\leq j\leq r+1$, the $j$-th derivative of
$g_{\cJ_{\intm}}$ is given by
\[
    g_{\cJ_{\intm}}^{(j)}(q)
    =
    \frac{1}{K_{\cJ_{\intm}}(u,u)}
    \sum_{k\in\cJ_{\intm}}
    \lambda_{t,k}^{-1}\phi_k(u)
    \frac{d^j}{dq^j}\{w_R(q)\phi_k(q)\}.
\]

Hence the key challenge of the proof is to show that, for $k\geq1$,
\begin{equation}\label{eq:proof_tapered_legendre_derivative}
    \left\|
        \frac{d^j}{dq^j}\{w_R(q)\phi_k(q)\}
    \right\|_\infty
    =
    \left\|
        \sum_{a=0}^j
        \binom{j}{a}
        w_R^{(a)}(q)\phi_k^{(j-a)}(q)
    \right\|_\infty
    \leq C_{R,r,\delta}k^j,
    \qquad 0\leq j\leq r+1.
\end{equation}

\paragraph{Bounding the derivative of $w_R$.} Fix $0\leq j\leq r+1$ and $0\leq a\leq j$. We have
\begin{align*}
    \big|w_R^{(a)}(q)\big|
    &=
    \frac{1}{\{u(1-u)\}^R}
    \left|
        \sum_{c=0}^a
        \binom{a}{c}
        \frac{d^c}{dq^c}q^R
        \frac{d^{a-c}}{dq^{a-c}}(1-q)^R
    \right|\\
    &\leq
    C_{R,a,\delta}
    \sum_{c=0}^a
    q^{R-c}(1-q)^{R-a+c}\\
    &=
    C_{R,a,\delta}
    \{q(1-q)\}^{R-a}
    \sum_{c=0}^a
    q^{a-c}(1-q)^c\\
    &\leq
    C_{R,a,\delta}
    \{q(1-q)\}^{R-a},
\end{align*}
where the first line follows from Leibniz's rule, the second line from $u\in[\delta,1-\delta]$, and the last from $q^{a-c}(1-q)^c\leq1$.

\paragraph{Bounding the derivative of $\phi_k$.} For the derivative bound on the shifted Legendre polynomials, fix an integer $0\leq\ell\leq r+1$. By  chain rule,
\[
    \big|\phi_k^{(\ell)}(q)\big|
    =
    2^\ell\sqrt{2k+1}
    \big|P_k^{(\ell)}(2q-1)\big|
    =
    2^\ell\sqrt{2k+1}
    \big|P_k^{(\ell)}(x)\big|,
\]
where $x=2q-1$. Since $P_k(-x)=(-1)^kP_k(x)$,
differentiating $\ell$ times with respect to $x$ gives
\[
    (-1)^\ell P_k^{(\ell)}(-x)
    =
    (-1)^kP_k^{(\ell)}(x),
\]
and subsequently
$P_k^{(\ell)}(-x)=(-1)^{k+\ell}P_k^{(\ell)}(x)$.
Hence
$\big|P_k^{(\ell)}(-x)\big|=\big|P_k^{(\ell)}(x)\big|$,
and it suffices to consider $0\leq x\leq1$.

Write $x=\cos\theta$, where $0\leq\theta\leq\pi/2$.
By Theorem~7.32.4 of \citet{szego75}, there exists a fixed
constant $c>0$ such that
\[
    \big|P_k^{(\ell)}(\cos\theta)\big|
    \leq
    C_\ell
    \begin{cases}
        k^{2\ell},
        &0\leq\theta\leq ck^{-1},\\[1mm]
        k^{\ell-1/2}\theta^{-\ell-1/2},
        &ck^{-1}\leq\theta\leq\pi/2,
    \end{cases}
\]
for all sufficiently large $k$. Suppose
$0\leq\theta\leq ck^{-1}$. Since $\sin\theta\leq\theta$,
we have
\[
    (\sin\theta)^{\ell+1/2}
    \big|P_k^{(\ell)}(\cos\theta)\big|
    \leq
    C_\ell\theta^{\ell+1/2}k^{2\ell}
    \leq
    C_\ell(ck^{-1})^{\ell+1/2}k^{2\ell}
    \leq
    C_\ell k^{\ell-1/2}.
\]
If instead $ck^{-1}\leq\theta\leq\pi/2$, again using
$\sin\theta\leq\theta$, we have
\[
    (\sin\theta)^{\ell+1/2}
    \big|P_k^{(\ell)}(\cos\theta)\big|
    \leq
    C_\ell(\sin\theta)^{\ell+1/2}
    k^{\ell-1/2}\theta^{-\ell-1/2}
    =
    C_\ell k^{\ell-1/2}
    \left(\frac{\sin\theta}{\theta}\right)^{\ell+1/2}
    \leq
    C_\ell k^{\ell-1/2}.
\]
Combining the two cases and noting
$\sin\theta=\sqrt{1-x^2}$ gives
\[
    (1-x^2)^{\ell/2+1/4}
    \big|P_k^{(\ell)}(x)\big|
    \leq C_\ell k^{\ell-1/2},
    \qquad 0\leq x\leq1.
\]
It follows that
\begin{equation}\label{eq:weighted_legendre_derivative}
    (1-x^2)^{\ell/2+1/4}
    \big|P_k^{(\ell)}(x)\big|
    \leq C_\ell k^{\ell-1/2},
    \qquad -1\leq x\leq1,
\end{equation}
for all sufficiently large $k$. For each of the finitely many
remaining positive integers $k$, the function
\[
    x\mapsto
    (1-x^2)^{\ell/2+1/4}
    \big|P_k^{(\ell)}(x)\big|
\]
is continuous on $[-1,1]$ and is therefore bounded.
After increasing $C_\ell$ if necessary,
\eqref{eq:weighted_legendre_derivative} holds for every $k\geq1$.

Now, since $1-x^2=4q(1-q)$,
\eqref{eq:weighted_legendre_derivative} gives
\begin{align*}
    \{q(1-q)\}^{\ell/2+1/4}
    \big|\phi_k^{(\ell)}(q)\big|
    &=
    2^{-1/2}\sqrt{2k+1}
    (1-x^2)^{\ell/2+1/4}
    \big|P_k^{(\ell)}(x)\big|\\
    &\leq
    C_\ell\sqrt{2k+1}k^{\ell-1/2}
    \leq C_\ell k^\ell.
\end{align*}

Setting $\ell=j-a$, we have for $0<q<1$ that
\[
    \big|w_R^{(a)}(q)\phi_k^{(j-a)}(q)\big|
    \leq
    C_{R,j,\delta}k^{j-a}
    \{q(1-q)\}^{R-a-(j-a)/2-1/4}.
\]
Note that the exponent of $q(1-q)$ is positive because
\[
    R-a-\frac{j-a}{2}-\frac14
    =
    R-\frac{j+a}{2}-\frac14
    \geq R-j-\frac14
    \geq \frac34,
\]
where the last inequality follows from $R\geq r+2$
and $j\leq r+1$. Since $0\leq q(1-q)\leq1/4$,
it follows that
\[
    \big|w_R^{(a)}(q)\phi_k^{(j-a)}(q)\big|
    \leq C_{R,j,\delta}k^{j-a}
    \leq C_{R,j,\delta}k^j.
\]
The same inequality holds at $q=0$ and $q=1$ by continuity.
Therefore, we have
\[
    \left\|
        \frac{d^j}{dq^j}\{w_R(q)\phi_k(q)\}
    \right\|_\infty
    \leq
    \sum_{a=0}^j
    \binom{j}{a}
    \left\|
        w_R^{(a)}\phi_k^{(j-a)}
    \right\|_\infty
    \leq C_{R,r,\delta}k^j,
    \qquad 0\leq j\leq r+1,
\]
and subsequently
\[
    \|g_{\cJ_{\intm}}^{(j)}\|_\infty
    \leq
    \frac{C_{R,r,\delta}}
    {K_{\cJ_{\intm}}(u,u)}
    \sum_{k\in\cJ_{\intm}}\lambda_{t,k}^{-1} \cdot m^j
    \qquad 0\leq j\leq r+1.
\]
by the derivative bound above, \eqref{eq:proof_uniform_legendre_bound}, and $k\leq m$. 

\paragraph{Showing $\sum_{k\in\cJ_{\intm}}\lambda_{t,k}^{-1}
\lesssim K_{\cJ_{\intm}}(u,u)$.}
Since $|\cJ_{\intm}|\geq W_0$, let
$N=\floor{|\cJ_{\intm}|/W_0}\geq1$ and partition
$\cJ_{\intm}$ into disjoint consecutive blocks
$\cJ_1,\ldots,\cJ_N$, taking the first $N-1$ blocks to have width $W_0$ and the last block to contain all remaining degrees. Then $W_0\leq|\cJ_\nu|<2W_0$ for every $\nu = 1,\ldots,N$.

For each $\nu$, let $k_\nu\equiv\max\cJ_\nu$. Applying Lemma~\ref{lem:inverse_eig_ratio} at $k_\nu$ with $\ell=k_\nu-k$ gives
\[
    \exp\left(-\frac{4\rho W_0}{1-\rho}\right)
    \leq
    \frac{\lambda_{t,k}^{-1}}{\lambda_{t,k_\nu}^{-1}}
    \leq 1,
    \qquad k\in\cJ_\nu,
\]
where the lower bound follows from $k_\nu-k<2W_0$ and $k_\nu\leq m\leq\rho t$. We then have
\begin{align*}
    \sum_{k\in\cJ_{\intm}}\lambda_{t,k}^{-1}
    &\leq
    \sum_{\nu=1}^{N}|\cJ_\nu|\lambda_{t,k_\nu}^{-1} \\
    &\lesssim
    \sum_{\nu=1}^{N}\lambda_{t,k_\nu}^{-1}
    \sum_{k\in\cJ_\nu}\phi_k(u)^2 \\
    &\lesssim
    \sum_{\nu=1}^{N}\sum_{k\in\cJ_\nu}
    \lambda_{t,k}^{-1}\phi_k(u)^2
    =
    K_{\cJ_{\intm}}(u,u),
\end{align*}
where the second inequality follows from
Lemma~\ref{lem:legendre_block} and the third from the eigenvalue comparison above. Hence,
\[
    \|g_{\cJ_{\intm}}^{(j)}\|_\infty
    \lesssim
    \frac{m^j}{K_{\cJ_{\intm}}(u,u)}
    \sum_{k\in\cJ_{\intm}}\lambda_{t,k}^{-1}
    \lesssim m^j,
    \qquad 0\leq j\leq r+1.
\]

\paragraph{Interpolation.} Finally, since $s=r+\alpha$, applying the interpolation inequality for the H\"older seminorm
$[f]_{C^\alpha([0,1])} \lesssim \|f\|_\infty^{1-\alpha}\|f'\|_\infty^\alpha$
with $f=g_{\cJ_{\intm}}^{(r)}$ gives
\[
    [g_{\cJ_{\intm}}^{(r)}]_{C^\alpha([0,1])}
    \lesssim
    (m^r)^{1-\alpha}(m^{r+1})^\alpha
    =m^s. 
\]
Together with the derivative bounds, this yields
\[
    \|g_{\cJ_{\intm}}\|_{C^s([0,1])}
    \lesssim m^s. \qedhere
\]
\end{proof}

\begin{proof}[Proof of Theorem \ref{thm:stat_lower_bound}]
Fix $m \in \cM_t$. 

\paragraph{Perturbed densities belong to the H\"{o}lder class.} We first show $p_\pm\in\mathcal H^s(L,p_{\max})$. 

Since $\int_0^1g_{\cJ_{\intm}}(q)\dd q=0$, both $p_{\pm,\Delta_m}$ integrate to one. Also, note that Lemma \ref{lem:holder_bound} and $\Delta_m\le c m^{-s}$ imply
\[
    \Delta_m\|g_{\cJ_{\intm}}\|_{C^s} \le Cc.
\]
Then we can choose $c>0$ 
sufficiently small so that 
\[
\|p_\pm\|_{C^s} \le 1+Cc \le L, \qquad  \|p_\pm\|_\infty \le 1+Cc \le p_{\max}.
\]
Moreover, since $\|g_{\cJ_{\intm}}\|_\infty\le \|g_{\cJ_{\intm}}\|_{C^s}$, we have 
\[
p_\pm(q) \ge 1-\Delta_m\|g_{\cJ_{\intm}}\|_\infty \ge 1-Cc \geq 1/2.
\]
Therefore $p_\pm\in\mathcal H^s(L,p_{\max})$.

\paragraph{Bounding the $\chi^2$ divergence.} Next we consider
\[
    q_\pm(x) \equiv \int_0^1 b_{t,x}(q)p_\pm(q)\dd q = \frac1{t+1}\pm\Delta_m(\cB_tg_{\cJ_{\intm}})(x),
\]
which are the actual observable laws in our setting. Since $p_{-,\Delta_m}\geq1/2$,
\[
    q_{-,\Delta_m}(x)
    =
    \int_0^1b_{t,x}(q)p_{-,\Delta_m}(q)\dd q
    \geq
    \frac{1}{2(t+1)}.
\]
It follows that
\begin{align*}
    \chi^2(q_{+,\Delta_m},q_{-,\Delta_m})
    &=
    \sum_{x=0}^t
    \frac{
        \{q_{+,\Delta_m}(x)-q_{-,\Delta_m}(x)\}^2
    }{
        q_{-,\Delta_m}(x)
    } \\
    &=
    4\Delta_m^2
    \sum_{x=0}^t
    \frac{(\cB_tg_{\cJ_{\intm}})(x)^2}{q_{-,\Delta_m}(x)} \\
    &\leq
    8(t+1)\Delta_m^2
    \sum_{x=0}^t(\cB_tg_{\cJ_{\intm}})(x)^2 \\
    &=
    8\Delta_m^2
    \ip{g_{\cJ_{\intm}}}{\cBD_tg_{\cJ_{\intm}}}_{L_2} \lesssim
    \frac{\Delta_m^2}{S_m(u)},
\end{align*}
where the last step follows from Lemma \ref{lem:operator_bound}. Hence for the product distributions, we have
\[
    1+
    \chi^2\left(
        q_{+,\Delta_m}^{\otimes n},
        q_{-,\Delta_m}^{\otimes n}
    \right) =
    \left\{
        1+\chi^2(q_{+,\Delta_m},q_{-,\Delta_m})
    \right\}^n \leq
    \exp\left\{
        n\chi^2(q_{+,\Delta_m},q_{-,\Delta_m})
    \right\}.
\]

Because $\Delta_m^2\le c^2S_m(u)/n$, 
\[
\chi^2(q_+^{\otimes n},q_-^{\otimes n})
\le
\exp(Cc^2)-1
\equiv \alpha_0.
\]
Choosing $c>0$ sufficiently small makes $\alpha_0$ a fixed absolute constant. Therefore, applying Le Cam's lemma \citep[Theorem 2.2]{tsybakov2008introduction} to $q_+^{\otimes n}$ and $q_-^{\otimes n}$ with separation $2\Delta_m$ yields
\[
\mathcal R_{n,t}
\ge
\frac{2\Delta_m}{8}\exp(-\alpha_0)
\gtrsim
\Delta_m.
\]

Maximizing over $m \in \cM_t$ proves the theorem.
\end{proof}

\subsection{Proofs for the small $t$ lower bound}
\label{sec:proofs_trial_only_lower_bound}

\begin{proof}[Proof of Lemma \ref{lem:trial_perturbation}]
The proof mostly follows from that of Lemma \ref{lem:holder_bound}.
For $j=0,\ldots,r+1$, we have 
\[
    \|g_{\cJ_{\sml}}^{(j)}\|_\infty
    \leq \frac{1}{\sum_{k\in\cJ_{\sml}}\phi_k(u)^2}
    \sum_{k\in\cJ_{\sml}}|\phi_k(u)|\left\|\frac{\dd^j}{\dd q^j}\{w_R\phi_k\}\right\|_\infty 
    \leq \frac{C}{W_0}\sum_{k\in\cJ_{\sml}}k^j 
    \lesssim t^j,
\]
where the second inequality follows from Lemma \ref{lem:legendre_block}, \eqref{eq:proof_uniform_legendre_bound} and \eqref{eq:proof_tapered_legendre_derivative}, and the third from $|\cJ_{\sml}|=W_0$ and $k\leq t+2R+W_0\leq(1+2R+W_0)t$ for every $t\geq1$ and $k\in\cJ_{\sml}$.

Therefore, applying the interpolation inequality with $f=g_{\cJ_{\sml}}^{(r)}$ yields 
\[
    [g_{\cJ_{\sml}}^{(r)}]_{C^\alpha} \lesssim (t^r)^{1-\alpha}(t^{r+1})^\alpha = t^s,
\]
and the desired follows. 
\end{proof}

\begin{proof}[Proof of Theorem \ref{thm:trial_lower_bound}]
We have by construction that $g_{\cJ_{\sml}}(u)=1$ and $\int_0^1g_{\cJ_{\sml}}(q)\dd q=0$. Further, by Lemma \ref{lem:trial_perturbation},
\[
    \Delta_t\norm{g_{\cJ_{\sml}}}_{C^s} \leq Cct^{-s}t^s = Cc.
\]
Thus, for a sufficiently small constant $c>0$, $p_{\pm,\Delta_t}(q)$ integrate to one, are nonnegative, and satisfy
\[
    \norm{p_{\pm,\Delta_t}}_{C^s}\leq L, \qquad \norm{p_{\pm,\Delta_t}}_\infty\leq p_{\max}.
\]
Therefore, $p_{\pm,\Delta_t}\in\Ht$.

Now, by construction, $g_{\cJ_{\sml}}$ is orthogonal to every polynomial in $\Pi_t$. It follows that
\[
    (\cB_tg_{\cJ_{\sml}})(x)
    =
    \int_0^1b_{t,x}(q)g_{\cJ_{\sml}}(q)\dd q
    =
    0,
    \qquad
    x=0,\ldots,t,
\]
and subsequently
\[
    q_{\pm,\Delta_t}(x)
    =
    \frac{1}{t+1}
    \pm
    \Delta_t(\cB_tg_{\cJ_{\sml}})(x)
    =
    \frac{1}{t+1}.
\]
Thus, the two mixing densities induce the same joint distribution of the observed sample.

For any estimator $\widehat p(u)$, we have
\[
    2ct^{-s} =
    |p_{+,\Delta_t}(u)-p_{-,\Delta_t}(u)| \leq
    |\widehat p(u)-p_{+,\Delta_t}(u)|
    +
    |\widehat p(u)-p_{-,\Delta_t}(u)|.
\]
Taking expectations under the joint distribution gives
\[
    \max_{\diamond\in\{+,-\}}
    \E_{p_{\diamond,\Delta_t}}
    |\widehat p(u)-p_{\diamond,\Delta_t}(u)|
    \geq
    ct^{-s}.
\]
Taking the infimum over all estimators gives the desired.
\end{proof}

\subsection{Proofs for the combined lower bound}
\label{sec:proofs_lower_bound_implications}

\begin{proof}[Proof of Corollary \ref{cor:lower_bound_bias_variance}]
For notational simplicity, we write $M \equiv\lfloor\rho t\rfloor$ and also
\[
    a_m\equiv m^{-s}, \qquad b_m\equiv\sqrt{\frac{S_m(u)}{n}}.
\]
$a_m$ is decreasing in $m$ and, since $S_m(u)$ is increasing, $b_m$ is increasing. We further define
\[
    L_* \equiv \max_{m_0 \leq m \leq M}\min(a_m,b_m),
    \qquad
    U_* \equiv \min_{m_0 \leq m \leq M}(a_m+b_m).
\]
\paragraph{Case 1: two sequences cross.} Suppose that there exists an $m \in \cM_t$ such that $b_m \geq a_m$, and let $m_*$ be the smallest such integer. If $m_*=m_0$, note that
\[
    S_{m_0}(u) 
    \leq C_\delta \sum_{k=0}^{m_0}\lambda_{t,k}^{-1}
    \leq C_\delta(m_0+1)\lambda_{t,m_0}^{-1} \leq C_\delta(m_0+1)\exp\left(\frac{2\rho m_0}{1-\rho}\right) \equiv C_{\delta, \rho, m_0},
\]
where the first inequality follows from \eqref{eq:proof_uniform_legendre_bound}, the second from monotonicity of the inverse eigenvalues, and the third from the lower bound in Lemma~\ref{lem:inverse_eig_ratio} with $m=\ell=m_0$ and $\lambda_{t,0}=1$, together with $m_0\leq\rho t$. It follows that
\[
    U_* \leq a_{m_0}+b_{m_0} \leq a_{m_0} + \sqrt{C_{\delta, \rho, m_0}} = (1+\sqrt{C_{\delta, \rho, m_0}}m_0^s)a_{m_0} \lesssim L_*.
\]

If instead $m_*>m_0$, then $b_{m_*-1} < a_{m_*-1}$ so 
\[
    b_{m_*-1} = \min(a_{m_*-1}, b_{m_*-1}) \leq L_*.
\]
Moreover, since $m_* \geq m_0 \geq 2$, we have $m_*/(m_*-1) \leq 2$ and subsequently
\[
    a_{m_*-1} = \left(\frac{m_*}{m_*-1}\right)^s a_{m_*} \leq 2^s a_{m_*}.
\]
It follows that 
\[
    U_* \le a_{m_*-1}+b_{m_*-1} \leq (2^s+1)L_* \lesssim L_*.
\]

\paragraph{Case 2: two sequences do not cross.}  Now suppose that $b_m < a_m$ for every $m \in \cM_t$, i.e., the two sequences do not cross. Then $b_M=\min(a_M,b_M) \leq L_*$ and
\[
    U_* \leq a_M+b_M = M^{-s} + b_M \leq M^{-s}+L_* \lesssim t^{-s}+L_* \lesssim \max(t^{-s}, L_*). \qedhere
\]
\end{proof}

\subsection{Proofs for attainability}
\begin{proof}[Proof of Proposition \ref{prop:singular_system}]
For $g = (g_0,\ldots,g_t) \in \R^{t+1}$, consider
\[
    (\cA_t g)(q) \equiv \sum_{x=0}^t g_x b_{t,x}(q).
\]
We have
\[
    \cA_t\psi_{k,t}
    = \lambda_{t,k}^{-1/2}\cA_t\cA_t^*\phi_k
    = \lambda_{t,k}^{-1/2}\cBD_t\phi_k
    = \sqrt{\lambda_{t,k}}\phi_k,
\]
where the first equality follows from the definition of $\psi_{k,t}$ and the adjoint relation $\cA_t^*=(t+1)\cB_t$, the second from the definition of the Bernstein--Durrmeyer operator
$\cBD_t=\cA_t\cA_t^*$, and the last from
Lemma \ref{lem:bernstein_durrmeyer_eigen}. Since
\[
    \E_p\{\psi_{k,t}(X)\mid Q=q\}
    = \sum_{x=0}^t \psi_{k,t}(x)b_{t,x}(q)
    = (\cA_t\psi_{k,t})(q),
\]
the result follows.
\end{proof}

\begin{proof}[Proof of Lemma \ref{lem:filtered_legendre}]
Fix an integer $m\geq1$ and $f\in L^\infty([0,1])$.
First note that, for $0\leq k\leq2m$,
\begin{align*}
    \eta\left(\frac{k}{m}\right)
    &=
    \sum_{j=k}^{2m}
    \left\{
        \eta\left(\frac{j}{m}\right)
        -
        \eta\left(\frac{j+1}{m}\right)
    \right\}\\
    &=
    \sum_{j=k}^{2m}(j-k+1)
    \sum_{a=0}^2(-1)^a\binom{2}{a}
    \eta\left(\frac{j+a}{m}\right)\\
    &=
    \sum_{j=k}^{2m}
    \binom{j-k+2}{2}
    \sum_{a=0}^3(-1)^a\binom{3}{a}
    \eta\left(\frac{j+a}{m}\right),
\end{align*}
It follows that we can write
\[
    (V_mf)(q)
    =
    \sum_{j=0}^{2m}
    \left\{
        \sum_{a=0}^3(-1)^a\binom{3}{a}
        \eta\left(\frac{j+a}{m}\right)
    \right\}
    \int_0^1 H_j(q,v)f(v)\,\dd v,
\]
where
\[
    H_j(q,v)
    \equiv
    \sum_{k=0}^j
    \binom{j-k+2}{2}\phi_k(q)\phi_k(v).
\]
By \citet[Section~9.7(2)]{szego75},
$H_j(q,v)\geq0$ for all $q,v\in[0,1]$.
Moreover, orthonormality and $\phi_0\equiv1$ give
\[
    \int_0^1 H_j(q,v)\,\dd v
    =
    \binom{j+2}{2}.
\]
It follows that
\[
    \left|\int_0^1 H_j(q,v)f(v)\,\dd v\right|
    \leq
    \|f\|_\infty\int_0^1 H_j(q,v)\,\dd v
    =
    \binom{j+2}{2}\|f\|_\infty.
\]

Now, applying the fundamental theorem of calculus yields
\[
    \sum_{a=0}^3(-1)^a\binom{3}{a}
    \eta\left(\frac{j+a}{m}\right) =
    -\int_0^{1/m}\int_0^{1/m}\int_0^{1/m}
    \eta^{(3)}(j/m+v_1+v_2+v_3)
    \,\dd v_1\,\dd v_2\,\dd v_3,
\]
and subsequently
\begin{align*}
    \left|
        \sum_{a=0}^3(-1)^a\binom{3}{a}
        \eta\left(\frac{j+a}{m}\right)
    \right|
    &\leq
    \|\eta^{(3)}\|_\infty
    \int_0^{1/m}\int_0^{1/m}\int_0^{1/m}
    1\,\dd v_1\,\dd v_2\,\dd v_3\\
    &=
    m^{-3}\|\eta^{(3)}\|_\infty.
\end{align*} 
As a result, we have
\begin{align*}
    \|V_mf\|_\infty
    &\leq
    m^{-3}\|\eta^{(3)}\|_\infty
    \sum_{j=0}^{2m}\binom{j+2}{2}\|f\|_\infty\\
    &=
    m^{-3}\binom{2m+3}{3}
    \|\eta^{(3)}\|_\infty\|f\|_\infty
    \lesssim
    \|\eta^{(3)}\|_\infty\|f\|_\infty,
\end{align*}
uniformly over $m\geq1$. Since $\eta$ is fixed and
$\|\eta^{(3)}\|_\infty<\infty$, this proves the result.
\end{proof}

\begin{proof}[Proof of Theorem \ref{thm:upper_bound}]
We first analyze the bias and then the variance.

\paragraph{Bias.} 
Note that Proposition \ref{prop:singular_system} gives
\begin{align*}
    \E_p\{\widehat p_m(u)\} &= \E_p \left\{\sum_{k=0}^{2m} \eta(k/m) \frac{\psi_{k,t}(X)}{\sqrt{\lambda_{t,k}}}\phi_k(u)\right\} \\
    &= \E_p\left\{\sum_{k=0}^{2m} \eta(k/m)\phi_k(Q)\phi_k(u)\right\} \\ 
    &= \sum_{k=0}^{2m} \eta(k/m)\int_0^1\phi_k(q)p(q)dq\phi_k(u) = V_mp(u).
\end{align*}

Since $\eta(k/m)=1$ for $k\le m$ and $\{\phi_0,\phi_1,\ldots,\phi_m\}$ form an orthonormal basis for $\Pi_m$, $V_m$ reproduces every polynomial in $\Pi_m$, the space of polynomials on $[0,1]$ of degree at most $m$, i.e., $V_m\pi = \pi$. Thus we can decompose the bias as
\[
    V_mp - p = V_mp -V_m\pi + \pi - p = V_m(p - \pi) - (p - \pi).
\]
It follows that
\[
    \norm{V_mp-p}_{\infty}
    \le
    \bigg\{1 + \sup_{\norm{f}_\infty \leq 1} \norm{V_mf}_\infty \bigg\}\norm{p-\pi}_{\infty} \leq (1+C)\norm{p-\pi}_{\infty},
\]
where the last inequality follows from Lemma \ref{lem:filtered_legendre}. Taking the infimum over $\pi\in\Pi_m$ and applying Jackson's theorem for $C^s$ functions 
gives
\[
    \norm{V_mp-p}_{\infty} \lesssim m^{-s},
\]
which implies
$
    \abs{\E_p\{\widehat p_m(u)\} -p(u)}
    \lesssim m^{-s}.
$

\paragraph{Variance.} Next consider the variance and, for notational simplicity, write
\[
    \widehat p_m(u)=\frac1n \sum_{i=1}^n h_{m,u}(X_i), \qquad h_{m,u}(x)
    \equiv \sum_{k=0}^{2m}\eta(k/m)
    \frac{\phi_k(u)}{\sqrt{\lambda_{t,k}}}\psi_{k,t}(x).
\]
We have
\begin{align*}
\operatorname{Var}_p \{\widehat p_m(u)\} &\le \frac1n\E\{h_{m,u}(X)^2\} \\
&\le \frac{p_{\max}}{n(t+1)}\sum_{x=0}^t h_{m,u}(x)^2 \\
&= \frac{p_{\max}}{n}\sum_{k=0}^{2m}\eta(k/m)^2\frac{\phi_k(u)^2}{\lambda_{t,k}} \\
&\le \frac{p_{\max}}{n}\sum_{k=0}^{2m}\frac{\phi_k(u)^2}{\lambda_{t,k}},
\end{align*}
where the second line follows from
\begin{align*}
    \E_p\{h_{m,u}(X)^2\} 
    &= \sum_{x=0}^t h_{m,u}(x)^2 \int_0^1 b_{t,x}(q)p(q)\dd q \\
    &\leq p_{\max}\sum_{x=0}^t h_{m,u}(x)^2 \int_0^1 b_{t,x}(q)\dd q \\
    &= \frac{p_{\max}}{t+1}\sum_{x=0}^t h_{m,u}(x)^2,    
\end{align*}
and the third from
\[
    \frac1{t+1}\sum_{x=0}^t h_{m,u}(x)^2 = \sum_{k=0}^{2m}\eta(k/m)^2\frac{\phi_k(u)^2}{\lambda_{t,k}},
\]
which holds by the orthonormality of $\psi_{k,t}$.

Putting both results together, we obtain
\[
    \Big[
    \E_p\{\widehat p_m(u)-p(u)\}^2
    \Big]^{1/2}
    \leq |\operatorname{bias}\{\widehat p_m(u)\}| + \sqrt{\Var\{\widehat p_m(u)\}}
    \lesssim 
    m^{-s}+\sqrt{\frac{\sum_{k=0}^{2m}\lambda_{t,k}^{-1}\phi_k(u)^2}{n}}. \qedhere
\]
\end{proof}

\subsection{Proofs for the overall minimax rate}
\begin{proof}[Proof of Corollary \ref{cor:minimax_rate}]
Corollary \ref{cor:lower_bound_bias_variance} contains the lower bound result, so we prove the upper bound. 

Fix $m\in \cM_t$ and let $j\equiv\floor{m/2}$. While $2j\leq m\leq t$, we have $j\geq m/3$ since $m_0\geq2$. Hence Theorem \ref{thm:upper_bound} together with $\E|Z|\leq(\E Z^2)^{1/2}$ gives
\[
    \cR_{n,t}
    \lesssim
    j^{-s}
    +
    \sqrt{\frac{S_{2j}(u)}{n}}
    \lesssim
    m^{-s}
    +
    \sqrt{\frac{S_m(u)}{n}},
\]
where the second inequality follows from $j\geq m/3$ and the monotonicity of $S_m(u)$. Taking the minimum gives the desired.
\end{proof}

\begin{proof}[Proof of Corollary~\ref{cor:simplified_minimax_rate}]
By Corollary~\ref{cor:minimax_rate}, it is enough to analyze
\[
    \min_{m\in\cM_t}
    \left\{
        m^{-s}+\sqrt{\frac{S_m(u)}{n}}
    \right\}.
\]
\paragraph{Lower bound.} 
First, set $a_\rho\equiv (1-\rho)/(1+\rho)$ and $b_\rho\equiv 2/(1-\rho)$ for notational simplicity. Note that Lemma~\ref{lem:inverse_eig_ratio} with both
indices equal to $k$ and $\lambda_{t,0}=1$ yields
\[
    \exp\left(a_\rho\frac{k^2}{t}\right)
    \leq \lambda_{t,k}^{-1}
    \leq \exp\left(b_\rho\frac{k^2}{t}\right),
    \qquad 0\leq k\leq m\leq\rho t.
\]
Then, \eqref{eq:proof_uniform_legendre_bound} implies
\[
    S_m(u)
    =
    \sum_{k=0}^m\lambda_{t,k}^{-1}\phi_k(u)^2
    \lesssim
    \sum_{k=0}^m\exp\left(b_\rho\frac{k^2}{t}\right)
    \lesssim
    m\exp\left(b_\rho\frac{m^2}{t}\right).
\]
At the same time, we have
\[
    S_m(u)
    \geq
    \sum_{k=\lceil m/2\rceil}^m
    \lambda_{t,k}^{-1}\phi_k(u)^2
    \geq
    \exp\left(a_\rho\frac{m^2}{4t}\right)
    \sum_{k=\lceil m/2\rceil}^m\phi_k(u)^2
    \gtrsim
    m\exp\left(a_\rho\frac{m^2}{4t}\right),
\]
where the last step follows from Lemma~\ref{lem:legendre_block},
since Assumption~\ref{assm:constants_homo} ensures that
$\{\lceil m/2\rceil,\ldots,m\}$ has width
$\lfloor m/2\rfloor+1\geq W_0$. It follows that
\[
    m\exp\left(a'_\rho\frac{m^2}{t}\right)
    \lesssim S_m(u)
    \lesssim
    m\exp\left(b_\rho\frac{m^2}{t}\right), \qquad a'_\rho \equiv \frac{a_\rho}{4},
\]
uniformly over $m\in\cM_t$ and $u\in[\delta,1-\delta]$.
Since $S_m(u)\gtrsim m$, minimizing $m^{-s}+\sqrt{m/n}$
over $m>0$, together with $m^{-s}\geq(\rho t)^{-s}$ for
$m\in\cM_t$, gives
\[
    \cR_{n,t}
    \gtrsim
    \max\left\{t^{-s},n^{-s/(2s+1)}\right\}.
\]

Now consider the intermediate $t$ regime. Let $\ell_{n,t}$ be defined in \eqref{eq:lnt} so that
\[
    1\leq\ell_{n,t}\leq1+\log n\leq2t,
    \qquad
    \frac{t^{s+1/2}}{n}=e^{1-\ell_{n,t}}.
\]
Fix $C>0$, depending only on $\rho$, sufficiently large that
$a'_\rho C^2\geq1$, and consider any $m\in\cM_t$.
If $m\leq C\sqrt{t\ell_{n,t}}$, then
\[
    m^{-s}\geq C^{-s}(t\ell_{n,t})^{-s/2}.
\]
Otherwise, the lower bound on $S_m(u)$ gives
\begin{align*}
    \frac{S_m(u)}{n}(t\ell_{n,t})^s
    &\gtrsim
    C\frac{(t\ell_{n,t})^{s+1/2}}{n}
    \exp\left(a'_\rho C^2\ell_{n,t}\right) \\
    &=
    C\ell_{n,t}^{s+1/2}
    \exp\{1+(a'_\rho C^2-1)\ell_{n,t}\}
    \gtrsim1.
\end{align*}
Thus every $m\in\cM_t$ satisfies
\[
    m^{-s}+\sqrt{\frac{S_m(u)}{n}}
    \gtrsim(t\ell_{n,t})^{-s/2}.
\]
Taking the minimum over $m\in\cM_t$ proves the intermediate lower bound.

\paragraph{Upper bound.} 
Now, fix $\gamma>0$ with $\gamma\leq\rho/\sqrt{2}$ and $b_\rho\gamma^2\leq1/2$.

For the small $t$ regime, choose 
\[
    m=\max\{m_0,\floor{\gamma t}\} \in \cM_t,
\]
so that $m\asymp t$. Since $m^2/t\leq\rho m_0+\gamma^2t$ and $n\geq e^t$, we have
\[
    \frac{S_m(u)}{n}
    \lesssim
    t\exp\{-(1-b_\rho\gamma^2)t\}
    \lesssim t^{-2s}.
\]
Together with $m^{-s}\asymp t^{-s}$, this yields $\cR_{n,t} \lesssim t^{-s}$.

For the large $t$ regime, choose
\[
    m=\max\{m_0,\floor{\gamma n^{1/(2s+1)}}\} \in \cM_t,
\]
which satisfies $m\asymp n^{1/(2s+1)}$ and gives $m^2/t\leq\rho m_0+\gamma^2\lesssim1$.
Hence $S_m(u)\lesssim m$ and subsequently
\[
    m^{-s}+\sqrt{\frac{S_m(u)}{n}}
    \lesssim n^{-s/(2s+1)}.
\]

Finally, for the intermediate regime, choose
\[
    m=\max\{m_0,\floor{\gamma\sqrt{t\ell_{n,t}}}\} \in \cM_t,
\]
which satisfies $m\asymp\sqrt{t\ell_{n,t}}$.
Since $m^2/t\leq\rho m_0+\gamma^2\ell_{n,t}$ and
$t^{s+1/2}/n=e^{1-\ell_{n,t}}$, we have
\[
    \frac{m^{2s}S_m(u)}{n}
    \lesssim
    \ell_{n,t}^{s+1/2}\exp\{1-(1-b_\rho\gamma^2)\ell_{n,t}\}
    \lesssim1.
\]
As a result,
\[
    \cR_{n,t}
    \lesssim
    m^{-s}
    \asymp
    (t\ell_{n,t})^{-s/2}.
\]

Combining the bounds proves the corollary.
\end{proof}

\begin{proof}[Proof of Theorem~\ref{thm:fundamental_lower_bound}]
For $t\geq t_0$, the result follows from the lower bound in
Corollary~\ref{cor:simplified_minimax_rate}. The remaining bounded values of $t$ are covered by the small $t$ lower bound in Theorem~\ref{thm:trial_lower_bound} after adjusting the implicit constant.
\end{proof}

\subsection{Proofs for trial heterogeneity}
\begin{proof}[Proof of Proposition \ref{prop:aggregate_eigendecomposition}]
    If $k\leq t_i$, then Lemma \ref{lem:bernstein_durrmeyer_eigen} gives $\cBD_{t_i}\phi_k = \lambda_{t_i,k}\phi_k$. If instead $k>t_i$, then every Bernstein basis function $b_{t_i,x}$ belongs to $\Pi_{t_i}$, while $\phi_k$ is orthogonal to
    $\Pi_{t_i}$. It follows that
    \[
        (\cB_{t_i}\phi_k)(x)
        =
        \int_0^1
        b_{t_i,x}(q)\phi_k(q)\dd q
        =
        0,
        \qquad
        x=0,\ldots,t_i,
    \]
    and therefore $\cBD_{t_i}\phi_k=0 = \lambda_{t_i,k}\phi_k$. 
    Summing over
    $i$ gives
    \[
        \cBD_{\mathbf t}\phi_k = \sum_{i=1}^n \cBD_{t_i}\phi_k = \sum_{i=1}^n\lambda_{t_i,k}\phi_k = \cI_k\phi_k. \qedhere
    \]    
\end{proof}

\begin{proof}[Proof of Lemma \ref{lem:heterogeneous_perturbation}]
For $j=0,\ldots,r+1$, uniformly over $u\in[\delta,1-\delta]$,
\[
    \|g_{\cJ_{\intm}^{\mathrm{het}}}^{(j)}\|_\infty
    \leq
    \frac{
        \sum_{k\in\cJ_{\intm}^{\mathrm{het}}}
        |\phi_k(u)|
        \left\|
            \frac{\dd^j}{\dd q^j}\{w_R\phi_k\}
        \right\|_\infty
    }{
        \sum_{k\in\cJ_{\intm}^{\mathrm{het}}}
        \phi_k(u)^2
    } 
    \lesssim \frac{1}{m}\sum_{k\in\cJ_{\intm}^{\mathrm{het}}}k^j
    \lesssim m^j,
\]
where the second inequality follows from
Lemma~\ref{lem:legendre_block},
\eqref{eq:proof_uniform_legendre_bound}, and
\eqref{eq:proof_tapered_legendre_derivative}, 
and the last from $|\cJ_{\intm}^{\mathrm{het}}|=m$ and $k\lesssim m$ for every $k\in\cJ_{\intm}^{\mathrm{het}}$.

Applying the interpolation inequality with
$f=g_{\cJ_{\intm}^{\mathrm{het}}}^{(r)}$ yields
\[
    [g_{\cJ_{\intm}^{\mathrm{het}}}^{(r)}]_{C^\alpha([0,1])}
    \lesssim
    (m^r)^{1-\alpha}(m^{r+1})^\alpha
    =m^s. \qedhere
\]
\end{proof}

\begin{proof}[Proof of Lemma \ref{lem:heterogeneous_operator_bound}]
We know from the proof of Lemma~\ref{lem:operator_bound} that multiplication by
$w_R$ can lower a Legendre degree by at most $2R$. Since the smallest degree in $\cJ_{\intm}^{\mathrm{het}}$ is $m+2R+1$, $g_{\cJ_{\intm}^{\mathrm{het}}}$ contains only degrees greater than $m$. Thus, Proposition \ref{prop:aggregate_eigendecomposition} and the monotonicity of $\cI_\ell$ yield
\[
    \langle
        g_{\cJ_{\intm}^{\mathrm{het}}},
        \cBD_{\mathbf t} g_{\cJ_{\intm}^{\mathrm{het}}}
    \rangle_{L_2}
    \leq \cI_m
    \|g_{\cJ_{\intm}^{\mathrm{het}}}\|_{L_2}^2 
    \leq
    \frac{
        \cI_m\|w_R\|_\infty^2
    }{
        \sum_{k\in\cJ_{\intm}^{\mathrm{het}}}
        \phi_k(u)^2
    }
    \lesssim \frac{\cI_m}{m},
\]
where the second inequality follows from the definition of
$g_{\cJ_{\intm}^{\mathrm{het}}}$ and orthonormality,
and the last from Lemma~\ref{lem:legendre_block} 
together with $\|w_R\|_\infty\leq\{4\delta(1-\delta)\}^{-R}$.
\end{proof}

\begin{proof}[Proof of
Corollary~\ref{cor:heterogeneous_statistical_lower_bound}]
Fix $m\in\cM_{\mathbf t}$. Since
$\int_0^1g_{\cJ_{\intm}^{\mathrm{het}}}(q)\dd q=0$,
Lemma~\ref{lem:heterogeneous_perturbation} and the same argument
as in the proof of Theorem~\ref{thm:stat_lower_bound} give
\[
    p_{\pm,\Delta_m^{\mathrm{het}}}\in\Ht,
    \qquad
    p_{-,\Delta_m^{\mathrm{het}}}(q)\geq\frac12,
    \quad q\in[0,1],
\]
for $c>0$ sufficiently small.

Let $q_{i,m,\pm}$ denote the probability mass functions of $X_i$
under these alternatives. Then
\[
    \sum_{i=1}^n\chi^2(q_{i,m,+},q_{i,m,-})
    \lesssim
    (\Delta_m^{\mathrm{het}})^2
    \left\langle
        g_{\cJ_{\intm}^{\mathrm{het}}},
        \cBD_{\mathbf t}g_{\cJ_{\intm}^{\mathrm{het}}}
    \right\rangle_{L_2}
    \lesssim
    (\Delta_m^{\mathrm{het}})^2\frac{\cI_m}{m}
    \lesssim c^2,
\]
where the first inequality follows by applying the $\chi^2$
calculation in the proof of
Theorem~\ref{thm:stat_lower_bound} to each observation,
the second from Lemma~\ref{lem:heterogeneous_operator_bound},
and the last from the definition of $\Delta_m^{\mathrm{het}}$.
By tensorization,
\[
    \chi^2\left(
        \bigotimes_{i=1}^n q_{i,m,+},
        \bigotimes_{i=1}^n q_{i,m,-}
    \right)
    \leq
    \exp\left\{
        \sum_{i=1}^n\chi^2(q_{i,m,+},q_{i,m,-})
    \right\}-1
    \lesssim c^2,
\]
for $c>0$ sufficiently small. 

Since $g_{\cJ_{\intm}^{\mathrm{het}}}(u)=1$, the alternatives
differ at $u$ by $2\Delta_m^{\mathrm{het}}$. Fixing $c>0$
sufficiently small and applying Le Cam's two-point lemma
therefore gives
\[
    \cR_{n,\mathbf t}
    \gtrsim
    \Delta_m^{\mathrm{het}}
    \asymp
    \min\left\{
        m^{-s},
        \sqrt{\frac{m}{\cI_m}}
    \right\}.
\]
Maximizing over $m\in\cM_{\mathbf t}$ gives the desired bound.
\end{proof}

\begin{proof}[Proof of Corollary \ref{cor:heterogeneous_trial_lower_bound}]
The proof follows that of Theorem \ref{thm:trial_lower_bound} with the common trial parameter $t$ replaced by $T$.
\end{proof}

For the proofs of Theorem \ref{thm:heterogeneous_upper_bound} and Corollary \ref{cor:heterogeneous_minimax_rate}, we write
\[
    S_{m,\mathbf t}(u) \equiv \sum_{k=0}^m \frac{\phi_k(u)^2}{\cI_k},
\]
for notational simplicity, which reduces to $S_{m,\mathbf t}(u)=S_m(u)/n$ under trial homogeneity.

\begin{proof}[Proof of Theorem \ref{thm:heterogeneous_upper_bound}]
Since every $\widehat\theta_k$ is unbiased, we again have
\[
    \E_p\{\widehat p_m^{\mathrm{het}}(q)\}
    =
    (V_mp)(q).
\]
Lemma \ref{lem:filtered_legendre} and the same polynomial approximation argument as in the proof of Theorem \ref{thm:upper_bound} therefore yield
\[
    \norm{\E_p(\widehat p_m^{\mathrm{het}})-p}_\infty
    \lesssim
    m^{-s}.
\]

For the variance term, write
\[
    \widehat p_m^{\mathrm{het}}(u) = \sum_{i=1}^nh_{i,m,u}(X_i), \qquad h_{i,m,u}(x)
    \equiv
    \sum_{k=0}^{\min(2m,t_i)}
    \eta(k/m)
    \frac{\sqrt{\lambda_{t_i,k}}}{\cI_k}
    \phi_k(u)\psi_{k,t_i}(x),
\]
for notational simplicity. We have
\begin{align*}
    \Var_p\{\widehat p_m^{\mathrm{het}}(u)\}
    &\leq
    \sum_{i=1}^n
    \E_p\{h_{i,m,u}(X_i)^2\}\\
    &\leq
    p_{\max}
    \sum_{k=0}^{2m}
    \eta(k/m)^2
    \frac{\phi_k(u)^2}{\cI_k^2}
    \sum_{i:t_i\geq k}\lambda_{t_i,k}\\
    &=
    p_{\max}
    \sum_{k=0}^{2m}
    \eta(k/m)^2
    \frac{\phi_k(u)^2}{\cI_k}\\
    &\leq
    p_{\max}S_{2m,\mathbf t}(u).
\end{align*}
Combining the bias and variance bounds and $\E|Z|\leq(\E Z^2)^{1/2}$ proves the theorem.
\end{proof}

\begin{proof}[Proof of Corollary \ref{cor:heterogeneous_minimax_rate}]
The lower bound follows directly from Theorem~\ref{thm:heterogeneous_lower_bound}. For the upper bound, fix $m \in \cM_{\mathbf{t}}$ and, as in the proof of Corollary~\ref{cor:minimax_rate}, let $j\equiv\floor{m/2}$ so that we again have $j\geq m/3$ while $2j\leq m\leq T$. Hence, Theorem~\ref{thm:heterogeneous_upper_bound} gives
\[
    \cR_{n,\mathbf t}
    \lesssim
    j^{-s}+\sqrt{S_{2j,\mathbf t}(u)}.
\]
Since $\cI_k$ is nonincreasing and $2j\leq m$, we have
\[
    S_{2j,\mathbf t}(u)
    =
    \sum_{k=0}^{2j}\frac{\phi_k(u)^2}{\cI_k}
    \leq
    \frac{1}{\cI_m}\sum_{k=0}^m\phi_k(u)^2
    \lesssim
    \frac{m}{\cI_m},
\]
where the last inequality follows from
\eqref{eq:proof_uniform_legendre_bound}.
Combining both bounds and using $j\geq m/3$, it follows that
\[
    \cR_{n,\mathbf t}
    \lesssim
    m^{-s}+\sqrt{\frac{m}{\cI_m}}.
\]
Taking the minimum over $m \in \cM_{\mathbf{t}}$ gives the desired.
\end{proof}

\begin{proof}[Proof of Corollary~\ref{cor:polynomial_trial_profile}]
Suppose, for notational simplicity, that the trial parameters are sorted in decreasing order. 

\paragraph{Lower bound.} By the eigenvalue
formula and $\log(1+x)\geq x/(1+x)$ for $x\geq0$,
\[
    -\log\lambda_{t,m}
    \geq
    \sum_{k=1}^m\frac{2k}{t+k+1}
    \geq \frac{m^2}{3t},
    \qquad 1\leq m\leq t.
\]
Since $\lambda_{t,m}=0$ for $m>t$, the subsequent upper bound $\lambda_{t,m} \leq \exp\{-m^2/(3t)\}$ holds for every integer $m \geq 1$. Hence, for some fixed $c>0$, we have
\[
    \cI_m
    \leq
    \sum_{i=1}^n\exp\{-cm^2(i/n)^{1/\beta}\}
    \leq
    n\int_0^\infty\exp(-cm^2x^{1/\beta})\,\dd x
    \lesssim nm^{-2\beta},
\]
where the second inequality follows because the integrand is nonincreasing, and the last by substituting $x=m^{-2\beta}y$. Thus, we have for every $m\in\cM_{\mathbf t}$ that
\[
    m^{-s}+\sqrt{\frac{m}{\cI_m}}
    \gtrsim m^{-s}+n^{-1/2}m^{\beta+1/2}
    \geq n^{-s/(2s+2\beta+1)}.
\]
The last inequality follows from the first term when $m\leq n^{1/(2s+2\beta+1)}$ and from the second otherwise.

\paragraph{Upper bound.} Choose $m\asymp n^{1/(2s+2\beta+1)}$. Since $T\asymp n^{1/\beta}$, this degree belongs to
$\cM_{\mathbf t}$ for sufficiently large $n$. 
The condition $t_i\gtrsim(n/i)^{1/\beta}$ therefore implies that at least a constant multiple of $nm^{-2\beta}$ observations satisfy $t_i\geq m^2/\rho$ for sufficiently large $n$. For these observations, $m\leq m^2\leq\rho t_i$, so Lemma~\ref{lem:inverse_eig_ratio} gives
\[
    \lambda_{t_i,m}
    \geq
    \exp\left(
        -\frac{2}{1-\rho}\frac{m^2}{t_i}
    \right)
    \geq
    \exp\left(-\frac{2\rho}{1-\rho}\right)
    \gtrsim1.
\]
It follows that 
\[
    m^{-s}+\sqrt{\frac{m}{\cI_m}} \lesssim m^{-s}+n^{-1/2}m^{\beta+1/2}\asymp n^{-s/(2s+2\beta+1)}, 
\]
since $\cI_m\gtrsim nm^{-2\beta}$ at this degree. 
\end{proof}

\end{document}